\documentclass[
framed_theorems,
framed_proofs,
print,
tcolorbox_frames,
]{OCPreprint}

\usepackage{enumitem}
\setlist[enumerate,1]{label=\textup{(\roman*)}}

\let\cite\parencite

\newcommand\sd{\textup{d}}
\renewcommand\ds{\sd_{\textup{s}}}

\renewcommand\graph{\operatorname{gph}}

\newcommand\BPfs{\textup{BP($f$, s)}}
\newcommand\BPfw{\textup{BP($f$, w)}}
\newcommand\BPfcs{\textup{BP($f\conjugate$, s)}}
\newcommand\BPfcw{\textup{BP($f\conjugate$, w$\star$)}}
\newcommand\RSfs{\textup{RS($f$, s)}}
\newcommand\RSfw{\textup{RS($f$, w)}}
\newcommand\RSfcs{\textup{RS($f\conjugate$, s)}}
\newcommand\RSfcw{\textup{RS($f\conjugate$, w$\star$)}}

\newcommand\BPpf{\textup{BP($\partial f$)}}
\newcommand\RSpf{\textup{RS($\partial f$)}}
\newcommand\NC{\textup{NC}}

\begin{document}
\title[Weak-$\star$ second subderivatives via $\Gamma$-convergence]%
{Simplifying the computation of weak-\texorpdfstring{\boldmath$\star$}{*} second subderivatives of convex functionals via \texorpdfstring{$\Gamma$}{Gamma}-convergence}
\author[Wachsmuth]{%
	Gerd Wachsmuth%
	\footnote{%
		Brandenburgische Technische Universität Cottbus--Senftenberg,
		Institute of Mathematics,
		03046 Cott\-bus,
		Germany,
		\email{gerd.wachsmuth@b-tu.de},
		\url{https://www.b-tu.de/fg-optimale-steuerung},
		ORCID: \href{https://orcid.org/0000-0002-3098-1503}{0000-0002-3098-1503}%
	}
}

\maketitle

\begin{abstract}
	We prove (under some mild assumptions)
	that
	the strong-strong twice epidifferentiability of a convex functional
	is equivalent to
	the weak-$\star$ twice epidifferentiability of its conjugate functional.
	Further,
	the corresponding subderivatives are (up to scaling)
	conjugates of each other.
	This result can be used to facilitate
	the computation of weak-$\star$ second subderivatives,
	which are crucial ingredients
	in no-gap second-order optimality conditions.
	We also show that
	weak-$\star$ second subderivatives are invariant
	under suitable changes of the underlying
	function space setting.
\end{abstract}

\begin{keywords}
	weak-$\star$ second subderivatives,
	second-order optimality conditions,
	$\Gamma$-convergence
\end{keywords}

\begin{msc}
	\mscLink{49K27},
	\mscLink{49J52},
	\mscLink{46G05}
\end{msc}

\section{Introduction}
\label{sec:intro}
We are interested in
quadratic growth conditions and
second-order optimality conditions for optimization problems of type
\begin{equation}
	\label{eq:problem}
	\text{Minimize} \quad F(x) + G(x)
	\quad\text{w.r.t.}
	\quad x \in X.
\end{equation}
Here,
the function $F \colon \dom G \to \R$ satisfies certain smoothness properties
while the function $G \colon X \to (-\infty,\infty]$ is nonsmooth.
The space $X$ is assumed to be the (topological) dual space of
a (typically separable) Banach space $Y$.
In recent years,
it has been demonstrated that the
(necessary and sufficient)
second-order conditions of \eqref{eq:problem}
can be characterized by using the so-called
\emph{weak-$\star$ second subderivative of $G$},
see, e.g.,
\cite{ChristofWachsmuth2017:1,WachsmuthWachsmuth2021,BorchardWachsmuth2023,WachsmuthWalter2024,BorchardMeyerWachsmuth2026:1}.
However,
the application of this theory was not so easy
since the computation of this weak-$\star$ second subderivative
was rather delicate,
see, e.g., the computations in
\cite[Section~6.2]{ChristofWachsmuth2017:1},
\cite[Sections~3, 4]{WachsmuthWachsmuth2021},
\cite[Section~6]{WachsmuthWalter2024}
and
\cite[Section~4]{BorchardMeyerWachsmuth2026:1}.
From a historical perspective,
we mention that,
in finite-dimensional spaces,
second subderivatives were introduced by \cite{Rockafellar1985}.
Later, they were studied
(under various names) in reflexive spaces and Hilbert spaces
by, e.g., 
\cite{Do1992,Noll1995,Levy1999}.

In this work,
we consider the important situation that the function $G$
is the convex conjugate of
a proper, convex and lower semicontinuous function
$J \colon Y \to (-\infty,\infty]$.
Note that this is precisely the case if $G$ is proper, convex
and lower semicontinuous w.r.t.\ the weak-$\star$ topology of $X = Y\dualspace$,
see \cite[Theorem~2.113]{BonnansShapiro2000}.
We focus on the case that $Y$ is separable but not reflexive,
since this case appears frequently in the above works
on the weak-$\star$ second subderivative.
Our main result
\cref{thm:main_result}
shows that (up to scaling)
the weak-$\star$ second subderivative of $G$
equals the convex conjugate
of
the strong second subderivative of $J$
under certain assumptions.
This result can be used to considerably shorten the computations
in the above references,
see the discussion following \cref{thm:main_result}.

A crucial ingredient for the equivalence in \cref{thm:main_result}
is the inequality \eqref{eq:growth_of_G_without_F}.
In \cref{cor:NDC_weak_cts},
we prove that this inequality already follows
from a quadratic growth condition
if $h \mapsto F''(\bar x) [h,h]$
is sequentially weak-$\star$ continuous.
Here, the bilinear form $F''(\bar x)$
serves as a replacement for a second-order derivative
of $F$, see \cref{asm:standing_subderivatives}.
This continuity assumption on $F''(\bar x)$
holds in many applications and, consequently,
\cref{thm:main_result} is broadly applicable.

Finally,
in \cref{sec:all_are_equal} we prove that the
weak-$\star$ second subderivative
is invariant under changes of the underlying space
as long as a strengthening of the condition \eqref{eq:growth_of_G_without_F}
is satisfied.
We refer to \cref{thm:equivalent_subderivatives} for the precise
formulation of this result.

Let us put our work into perspective.

In the case of reflexive spaces $X$ and $Y$,
it has been noted in \cite[Theorem~2.5]{Do1992}
that
a proper, convex and lower semicontinuous function is strongly twice epidifferentiable
(see \cref{def:subderivative} below)
if and only if 
its conjugate function is strongly twice epidifferentiable.
The key to this result is Mosco's theorem,
see \cite{Mosco1971},
which shows that the
so-called Mosco epiconvergence of a sequence $\seq{f_k} \subset \Gamma_0(Y)$
is equivalent to the Mosco epiconvergence of the conjugates $\seq{f_k\conjugate} \subset \Gamma_0(X)$.
Similarly, owing to 
Attouch's theorem,
see \cite[Theorem~3.66]{Attouch1984},
\cite[Theorem~3.9]{Do1992} proves
that the strong twice epidifferentiability of $G$
is equivalent to the subdifferential $\partial G$
being proto-differentiable (and the derivative being a maximal monotone operator).

However, both of these results fail if the underlying space is not reflexive,
see \cite{BeerBorwein1990,AttouchBeer1993}.

This issue has been addressed in the recent contribution
\cite{CorreaPerezArosSantander2025}.
In this work, the authors have shown that
the $\Gamma$-convergence of a sequence $\seq{f_k} \subset \Gamma_0(Y)$
(w.r.t.\ strong convergence)
is equivalent to
the $\Gamma$-convergence of the conjugate sequence $\seq{f_k\conjugate} \subset \Gamma_0(X)$
(w.r.t.\ weak-$\star$ convergence)
under some additional, mild assumptions.
The goal of the current paper is the application
of the results of \cite{CorreaPerezArosSantander2025}
in order to simplify the computation of
weak-$\star$ second subderivatives.

The outline of the paper is as follows.
In \cref{sec:notation_and_prelim},
we fix the notation and collect some preliminaries.
In \cref{sec:epi,sec:subdifferentials},
we review (special cases of) the results of
\cite{CorreaPerezArosSantander2025}.
We state the results in a form such that they are applicable
to our application to weak-$\star$ second subderivatives.
For convenience of the reader, we also give simple proofs.
If the reader is already familiar with these results,
\cref{sec:epi,sec:subdifferentials} can be skipped.
In \cref{sec:SSC} we mainly review the theory of
weak-$\star$ second subderivatives and
the associated second-order optimality conditions.
We also state a novel result (\cref{thm:characterization_of_NDC})
together with two corollaries
in which we give new characterizations of the
so-called non-degeneracy condition.
The results stated in
\cref{sec:epi,sec:subdifferentials}
are applied to the theory of second subderivatives in \cref{sec:subderivatives}.
In particular, we prove that a functional is
weak-$\star$ twice epidifferentiable (see \cref{def:subderivative})
if and only if
its preconjugate functional is
strongly-strongly twice epidifferentiable (see \cref{def:strong-subderivative})
as long as the additional assumption \eqref{eq:growth_of_G_without_F}
is satisfied.
Finally,
we investigate the invariance of the weak-$\star$ second subderivative
w.r.t.\ the underlying space in \cref{sec:all_are_equal}.

\section{Notation and preliminaries}
\label{sec:notation_and_prelim}

\subsection{Notation}
\label{sec:notation}

Throughout this work, $Y$ denotes a (real) Banach space
and $X := Y\dualspace$ is the associated (topological) dual space.
By $\Gamma_0(Y)$ we denote the set of proper, convex and lower semicontinuous functions.
The closed ball in $X$ with center $\bar x \in X$ and radius $r \in (0,\infty]$
is denoted by $B^X_r(\bar x)$.

A fundamental concept in variational analysis is the so-called $\Gamma$-convergence.
Let $f_k, f \colon Y \to [-\infty,\infty]$, $k \in \N$,
be given functions.
We say that the sequence $\seq{f_k}$ $\Gamma$-converges towards $f$
if
\begin{itemize}
	\item
		for all sequences $\seq{y_k} \subset Y$ and $y \in Y$ with $y_k \to y$ in $Y$,
		we have $f(y) \le \liminf_{k \to \infty} f_k(y_k)$
		and
	\item
		for all $y \in Y$ there exists a sequence $\seq{y_k} \subset Y$ satisfying $y_k \to y$ in $Y$
		and
		$\limsup_{k \to \infty} f_k(y_k) \le f(y)$.
\end{itemize}
Formally,
we will refer to the first property as the ``bound property'',
whereas the second one is the ``existence of a recovery sequence''.
Similarly, we can define $\Gamma$-convergence w.r.t.\ the weak topology of $Y$
and $\Gamma$-convergence of the conjugate functions $f_k\conjugate$ defined on the dual space $X = Y\dualspace$
w.r.t.\ strong or weak-$\star$ convergence.

For later reference, we abbreviate the various bound properties
appearing in these modes of convergence by
\begin{alignat*}{3}
	\tag{\BPfs}
	\label{BPfs}
	& \forall \seq{y_k} \subset Y : &&\forall y \in Y :
	&y_k \to y \; &\Rightarrow \; f(y) \le \liminf_{k \to \infty} f_k(y_k)
	,\\
	\tag{\BPfw}
	\label{BPfw}
	& \forall \seq{y_k} \subset Y : &&\forall y \in Y :
	&y_k \weakly y \; &\Rightarrow \; f(y) \le \liminf_{k \to \infty} f_k(y_k)
	,\\
	\tag{\BPfcs}
	\label{BPfcs}
	& \forall \seq{x_k} \subset X : {}&&\forall x \in X : {}
	&x_k \to x \; &\Rightarrow \; f\conjugate(x) \le \liminf_{k \to \infty} f_k\conjugate(x_k)
	,\\
	\tag{\BPfcw}
	\label{BPfcw}
	& \forall \seq{x_k} \subset X : &&\forall x \in X :
	&x_k \weaklystar x \; &\Rightarrow \; f\conjugate(x) \le \liminf_{k \to \infty} f_k\conjugate(x_k)
	.
\end{alignat*}

Similarly, the existence of a recovery sequence is abbreviated as
\begin{align*}
	\tag{\RSfs}
	\label{RSfs}
	&
	\forall y \in Y : \exists \seq{y_k} \subset Y :
	&&
	y_k \to y \land \limsup_{k \to \infty} f_k(y_k) \le f(y)
	,\\
	\tag{\RSfw}
	\label{RSfw}
	&
	\forall y \in Y : \exists \seq{y_k} \subset Y :
	&&
	y_k \weakly y \land \limsup_{k \to \infty} f_k(y_k) \le f(y)
	,\\
	\tag{\RSfcs}
	\label{RSfcs}
	&
	\forall x \in X : \exists \seq{x_k} \subset X :
	&&
	x_k \to x \land \limsup_{k \to \infty} f_k\conjugate(x_k) \le f\conjugate(x)
	,\\
	\tag{\RSfcw}
	\label{RSfcw}
	&
	\forall x \in X : \exists \seq{x_k} \subset X :
	&&
	x_k \weaklystar x \land \limsup_{k \to \infty} f_k\conjugate(x_k) \le f\conjugate(x)
	.
\end{align*}

\subsection{Preliminaries}
\label{sec:prelim}

We collect some simple, auxiliary results.
The first lemma shows that the bound property can also be invoked on subsequences.
\begin{lemma}
	\label{lem:aux_lower_bound}
	Let $f_k, f \colon Y \to [-\infty,\infty]$,
	$k \in \N$,
	be given
	satisfying \eqref{BPfs}.
	Suppose that a sequence $\seq{y_k}$ is given
	such that $y_{k_n} \to y$ in $Y$ along a subsequence.
	Then,
	$f(y) \le \liminf_{n \to \infty} f_{k_n}(y_{k_n})$.

	The same assertion holds for weak convergence
	and for weak-$\star$ in the dual space $X$.
\end{lemma}
\begin{proof}
	We modify the sequence as follows.
	If $k \in \N$ does not appear in the subsequence,
	i.e., $k \ne k_n$ for all $n \in \N$,
	we set $y_k := y$.
	Note that this ensures $y_k \to y$ for the entire sequence.
	Consequently,
	\begin{equation*}
		f(y)
		\le
		\liminf_{k \to \infty} f_k(y_k)
		\le
		\liminf_{n \to \infty} f_{k_n}(y_{k_n})
	\end{equation*}
	and this finishes the proof.

	Clearly, the same argument can be used for the conjugate functions
	and for weak-$\star$ convergence.
\end{proof}

We utilize the following result from
\cite[Lemma~1.15]{Attouch1984},
see also \cite[Lemma~7.1]{CorreaPerezArosSantander2025},
which greatly simplifies the selection of diagonal sequences.
\begin{lemma}
	\label{lem:technical_lemma_from_attouch}
	Let $\seq{\alpha_{l,k}}_{l,k \in \N} \subset [-\infty,\infty]$
	be given.
	Then, there is a sequence $\seq{l_k}_{k \in \N} \subset \N$
	with $l_k \to \infty$ and
	\begin{equation*}
		\limsup_{k \to \infty} \alpha_{l_k, k}
		\le
		\limsup_{l \to \infty} \limsup_{k \to \infty} \alpha_{l,k}
		.
	\end{equation*}
\end{lemma}
For convenience, we provide a short proof
inspired by the proof of
\cite[Lemma~2.12]{ChristofWachsmuth2017:3}.
\begin{proof}
	We note that for any sequence $\seq{\beta_k} \subset [-\infty,\infty]$,
	we have
	\begin{equation*}
		\exp\parens*{
			\limsup_{k \to \infty} \beta_k
		}
		=
		\limsup_{k \to \infty} \braces*{ \exp(\beta_k) }
		.
	\end{equation*}
	Consequently,
	by replacing $\alpha_{l,k}$ with $\exp(\alpha_{l,k})$,
	it is sufficient to consider the case that $\alpha_{l,k} \ge 0$
	for all $l,k \in \N$.

	We define
	$\alpha_l := \limsup_{k \to \infty} \alpha_{l,k}$
	and
	$\alpha := \limsup_{l \to \infty} \alpha_l$.
	It is sufficient to consider the case $\alpha < \infty$.
	Thus, there exists $l_0 \in \N$
	such that $\alpha_l < \infty$ for all $l \ge l_0$.
	We define $K_{l_0} := 1$
	and for all $l > l_0$ we choose $K_l > K_{l-1}$
	such that
	\begin{equation}
		\label{eq:in_technical_proof}
		\alpha_{l,k}
		\le
		\alpha_l + \frac{1}{l}
		\qquad
		\forall k \ge K_l
		.
	\end{equation}
	Next, we define $l_k := \sup\set{l \ge l_0 \given k \ge K_l}$.
	Then, for all $k \in \N$ we have
	$l_k \in \N$
	and $l_k \to \infty$
	since $\seq{K_l}$ is strictly increasing,
	and
	$k \ge K_{l_k}$ by definition.
	Consequently,
	for $k$ large enough, we have $l_k > l_0$ and
	\eqref{eq:in_technical_proof} implies
	\begin{equation*}
		\alpha_{l_k,k} \le \alpha_{l_k} + \frac{1}{l_k}
		.
	\end{equation*}
	Hence,
	\begin{align*}
		\limsup_{k \to \infty} \alpha_{l_k,k}
		&\le
		\limsup_{k \to \infty}
		\braces*{
			\alpha_{l_k} + \frac{1}{l_k}
		}
		=
		\limsup_{k \to \infty} \alpha_{l_k}
		\\&
		\le
		\limsup_{l \to \infty} \alpha_l
		=
		\limsup_{l \to \infty} \limsup_{k\to\infty} \alpha_{l,k}
		.
	\end{align*}
	This finishes the proof.
\end{proof}

Finally, we recall the celebrated Rockafellar integration formula
which allows us to reconstruct a convex function from its subdifferential,
see \cite{Rockafellar1966,Rockafellar1970:2}.
For an elementary proof (based on Ekeland's variational principle),
we refer to \cite{IvanovZlateva2008}.
\begin{theorem}
	\label{thm:rockafellar_integration}
	Let $f \in \Gamma_0(Y)$ be given.
	Then,
	\begin{equation*}
		f(y)
		=
		f(y_0)
		+
		\sup\set*{
			\sum_{i = 0}^{n} \dual{x_i}{y_{i+1} - y_i}_Y
			\given
			\begin{aligned}
				& n \ge 1, y_{n+1} := y,
				\\
				&
				\forall i = 1,\ldots,n :
				(y_i,x_i) \in \graph\partial f
			\end{aligned}
		}
	\end{equation*}
	holds
	for all $(y_0, x_0) \in \graph \partial f$
	and
	all $y \in Y$.
\end{theorem}

\section{Relations between the bound properties and the existence of recovery sequences}
\label{sec:epi}

In this section, we want to study
links between the convergence properties of a sequence $\seq{f_k}$
to those of the sequence of the conjugates $\seq{f_k\conjugate}$.
We point out that the results of this section follow from those
in \cite[Section~5]{CorreaPerezArosSantander2025}.
However, we believe that our presentation increases
the understanding.
In particular, our proof of \cref{prop:weakstar_dual_recovery}
is simpler than the corresponding proof in
\cite[Section~8.1]{CorreaPerezArosSantander2025}
and follows the lines of the proof of \cref{prop:strong_dual_recovery},
which is basically the result underlying Mosco's theorem
on the continuity of the Young--Fenchel transform
w.r.t.\ Mosco convergence.

Throughout this section, we consider proper, convex and lower semicontinuous functions.
\begin{assumption}[Standing assumption]
	\label{asm:standing_epiconvergence}
	We have
	$\seq{f_k} \subset \Gamma_0(Y)$ and $f \in \Gamma_0(Y)$.
\end{assumption}
Note that this immediately implies
$f_k\conjugate, f\conjugate \in \Gamma_0(X)$.

It is well known that the existence of recovery sequences
implies the lower estimates for the conjugate functions
by using appropriate topologies.
\begin{lemma}\leavevmode
	\label{lem:lower_bounds}
	\begin{enumerate}[label=(\alph*)]
		\item
			\label{lem:lower_bounds:1}
			\eqref{RSfs} implies \eqref{BPfcw}.
		\item
			\label{lem:lower_bounds:2}
			\eqref{RSfw} implies \eqref{BPfcs}.
		\item
			\label{lem:lower_bounds:3}
			\eqref{RSfcs} implies \eqref{BPfw}.
		\item
			\label{lem:lower_bounds:4}
			\eqref{RSfcw} implies \eqref{BPfs}.
	\end{enumerate}
\end{lemma}
For completeness, we provide the easy proof.
\begin{proof}
	We only prove \ref{lem:lower_bounds:1}.
	The other assertions follow by exactly the same argument
	(and $f = f\biconjugate$ for \ref{lem:lower_bounds:3} and \ref{lem:lower_bounds:4}).

	Let a sequence $\seq{x_k} \subset X$ with $x_k \weaklystar x$ in $X$ be given.
	For an arbitrary $y \in Y$
	let $\seq{y_k} \subset Y$ be the associated recovery sequence,
	i.e.,
	$y_k \to y$ in $Y$ and $\limsup_{k \to \infty} f_k(y_k) \le f(y)$.
	From the Fenchel--Young inequality, we infer
	\begin{equation*}
		f_k\conjugate(x_k)
		\ge
		\dual{x_k}{y_k}_Y
		-
		f_k(y_k).
	\end{equation*}
	From the convergence properties of the sequences we get
	$\dual{x_k}{y_k}_Y \to \dual{x}{y}_Y$.
	Consequently,
	\begin{equation*}
		\liminf_{k \to \infty} f_k\conjugate(x_k)
		\ge
		\dual{x}{y}_Y
		-
		\limsup_{k \to \infty} f_k(y_k)
		\ge
		\dual{x}{y}_Y
		-
		f(y)
		.
	\end{equation*}
	Since this inequality holds for all $y \in Y$,
	we can take the supremum w.r.t.\ $y \in Y$
	and get
	\begin{equation*}
		\liminf_{k \to \infty} f_k\conjugate(x_k)
		\ge
		\sup\set*{
			\dual{x}{y}_Y
			-
			f(y)
			\given
			y \in Y
		}
		=
		f\conjugate(x)
		.
	\end{equation*}
	This finishes the proof.
\end{proof}

Interestingly, it is also possible to infer the existence of recovery sequences
under the assumption of the bound property
in certain cases.
We start with the easiest situation by assuming reflexivity.
This result is essentially at the heart of Mosco's theorem.
The proof is inspired by the proof of
\cite[Theorem~5.53]{HuPapageorgiou1997}.

\begin{proposition}
	\label{prop:strong_dual_recovery}
	We assume that the following conditions are satisfied.
	\begin{enumerate}[label=(\arabic*)]
		\item
			The space $Y$ is reflexive.
		\item
			There exists $C \ge 0$ such that
			$f_k \ge - C \parens{ \norm{\cdot}_Y + 1}$
			and
			$f_k\conjugate \ge - C \parens{ \norm{\cdot}_X + 1}$
			hold for all $k \in \N$.
	\end{enumerate}
	Then,
	\eqref{BPfw}
	implies
	\eqref{RSfcs}.
\end{proposition}

Before diving into the proof,
we explain the underlying idea.
In order to find a sequence $\seq{x_k}$ with $x_k \to x$
such that $\seq{f_k\conjugate(x_k)}$ is small enough,
we just take minimizers of
$f_k\conjugate + \frac{l}{2} \norm{\cdot - x}_X^2$
with an appropriate coupling of $l$ and $k$.

\begin{proof}
	Let $x \in X$ be arbitrary.
	It is sufficient to consider $f\conjugate(x) < \infty$.
	Given $k,l \in \N$, we select
	\begin{equation*}
		x_{l,k}
		\in
		\argmin_{\tilde x \in X}
		\parens*{
			f_k\conjugate( \tilde x )
			+
			\frac{l}{2} \norm{ \tilde x - x }_X^2
		}
		.
	\end{equation*}
	Note that the existence of a minimizer follows from coercivity and reflexivity.
	Due to the reflexivity, the optimality implies the existence of $y_{l,k} \in Y \isometric X\dualspace$
	such that
	\begin{equation*}
		y_{l,k} \in \partial f_k\conjugate( x_{l,k} )
		,
		\qquad
		-y_{l,k} \in \frac{l}{2} \partial\parens*{\norm{\cdot}_X^2} ( x_{l,k} - x )
		.
	\end{equation*}
	Note that the second inclusion implies
	\begin{equation*}
		\norm{y_{l,k}}_Y
		=
		l \norm{x_{l,k} - x}_X
		\quad\text{and}\quad
		\dual{x_{l,k} - x}{-y_{l,k}}_Y
		=
		l \norm{x_{l,k} - x}_X^2
		.
	\end{equation*}
	Together with the first inclusion and the Fenchel--Young identity
	we get
	\begin{align*}
		f_k( y_{l,k} ) + f_k\conjugate( x_{l,k} )
		=
		\dual{x_{l,k}}{y_{l,k}}_Y
		&=
		\dual{x}{y_{l,k}}_Y
		+
		\dual{x_{l,k} - x}{y_{l,k}}_Y
		\\&
		=
		\dual{x}{y_{l,k}}_Y
		-
		l \norm{ x_{l,k} - x }_X^2
		.
	\end{align*}
	Consequently,
	the lower bounds on $f_k$ and $f_k\conjugate$ yield
	\begin{align*}
		l \norm{ x_{l,k} - x }_X^2
		&=
		\dual{x}{y_{l,k}}_Y
		-
		f_k( y_{l,k} )
		-
		f_k\conjugate( x_{l,k} )
		\\&
		\le
		\norm{x}_X \norm{ y_{l,k} }_Y
		+
		C( \norm{y_{l,k}}_Y + 1)
		+
		C( \norm{x_{l,k}}_X + 1)
		\\&
		=
		\norm{x}_X l \norm{ x_{l,k} - x }_X
		+
		C( l \norm{x_{l,k} - x}_X + 1)
		+
		C( \norm{x_{l,k}}_X + 1)
		.
	\end{align*}
	This shows that $\norm{x_{l,k}}_X$ is uniformly bounded.
	Hence, for each fixed $l \in \N$,
	the sequence $\seq{y_{l,k}}_k$ is bounded.
	Given $l \in \N$, we first choose a subsequence which
	realizes
	\begin{equation*}
		\limsup_{k \to \infty}\parens*{
			f_k\conjugate(x_{l,k})
			+
			l \norm{x_{l,k} - x}_X^2
		}
	\end{equation*}
	and,
	owing to the reflexivity of $Y$,
	this subsequence has a further subsequence
	such that $y_{l,k_n} \weakly y_l$ in $Y$.
	Thus,
	\begin{align*}
		\limsup_{k \to \infty}\parens*{
			f_k\conjugate(x_{l,k})
			+
			l \norm{x_{l,k} - x}_X^2
		}
		&=
		\lim_{n \to \infty} \parens*{
			f_{k_n}\conjugate(x_{l,k_n})
			+
			l \norm{x_{l,k_n} - x}_X^2
		}
		\\&
		=
		\lim_{n \to \infty} \parens*{
			\dual{x}{y_{l,k_n}}_Y - f_{k_n}(y_{l,k_n})
		}
		\\&
		\le
		\dual{x}{y_l}_Y - f(y_l)
		\le
		f\conjugate(x)
		<
		\infty
		,
	\end{align*}
	where the first inequality follows from \cref{lem:aux_lower_bound}.
	Owing to \cref{lem:technical_lemma_from_attouch},
	we find a sequence $\seq{l_k} \subset \N$ with $l_k \to \infty$
	as $k \to \infty$ and
	\begin{equation}
		\label{eq:diagonal_sequence_inequality}
		\limsup_{k \to \infty}\parens*{
			f_k\conjugate(x_{l_k,k})
			+
			l_k \norm{x_{l_k,k} - x}_X^2
		}
		\le
		f\conjugate(x)
		<
		\infty
		.
	\end{equation}
	We define $x_k := x_{l_k,k}$.
	Using the lower bound for $f_k\conjugate$, we find
	\begin{align*}
		f_k\conjugate(x_k)
		&\ge
		-C (\norm{ x_k }_X + 1 )
		\ge
		-C (\norm{ x_k - x }_X + \norm{x}_X + 1 )
		\\&
		\ge
		-\frac{l_k}{2} \norm{ x_k - x }_X^2
		-\frac{C^2}{2 l_k}
		-C (\norm{x}_X + 1 )
		.
	\end{align*}
	Inserting this into \eqref{eq:diagonal_sequence_inequality}
	and using $l_k \to \infty$ shows $\norm{x_k - x}_X \to 0$.
	Similarly,
	\eqref{eq:diagonal_sequence_inequality} implies
	\begin{equation*}
		\limsup_{k \to \infty}
		f_k\conjugate(x_k)
		\le
		\limsup_{k \to \infty}\parens*{
			f_k\conjugate(x_k)
			+
			l_k \norm{x_k - x}_X^2
		}
		\le
		f\conjugate(x)
		.
	\end{equation*}
	Hence, $\seq{x_k}$ is the desired recovery sequence.
\end{proof}
We mention that the reflexivity of $Y$
is crucial in the proof.
Without it, we cannot extract a weakly convergent subsequence of $\seq{y_{l,k}}_k$
which is needed to apply \eqref{BPfw}.
Further, without the assumption that $Y$ is reflexive,
the optimality condition would provide a subgradient from the dual space $X\dualspace$
which is a proper superspace of $Y$.

We briefly mention that the lower bounds on $f_k$ and $f_k\conjugate$
are not restrictive, see, e.g.,
\cref{lem:uniformly_proper_bound,lem:bound_on_fk_conjugate,lem:bound_on_fk}
below.

Next, we want to provide two variations of \cref{prop:strong_dual_recovery}.
First, we show that
\eqref{BPfcw} implies \eqref{RSfs}
by almost the same argument.
Interestingly,
this can be done without using reflexivity of the space $Y$.
\begin{proposition}
	\label{prop:strong_primal_recovery}
	We assume that the following conditions are satisfied.
	\begin{enumerate}[label=(\arabic*)]
		\item
			The space $Y$ is reflexive or separable.
		\item
			There exists $C \ge 0$ such that
			$f_k \ge - C \parens{ \norm{\cdot}_Y + 1}$
			and
			$f_k\conjugate \ge - C \parens{ \norm{\cdot}_X + 1}$
			hold for all $k \in \N$.
	\end{enumerate}
	Then, \eqref{BPfcw} implies \eqref{RSfs}.
\end{proposition}
\begin{proof}
	Let $y \in Y$ be arbitrary.
	It is sufficient to consider $f(y) < \infty$.
	Given $k,l \in \N$, 
	the function
	\begin{equation*}
		Y \ni \tilde y
		\mapsto
			f_k( \tilde y )
			+
			\frac{l}{2} \norm{ \tilde y - y }_Y^2
	\end{equation*}
	is bounded from below.
	Consequently,
	we can invoke Ekeland's variational principle to find
	$y_{l,k} \in Y$ such that
	\begin{equation*}
		0
		\in
		\partial f_k( y_{l,k} )
		+
		\frac{l}{2} \partial \parens*{ \norm{\cdot}_Y^2 }( y_{l,k} - y)
		+
		B^X_{1/k}(0)
		.
	\end{equation*}
	This implies the existence of
	$x_{l,k} \in X = Y\dualspace$
	and
	$\xi_{l,k} \in B^X_{1/k}(0) \subset X$
	such that
	\begin{equation*}
		x_{l,k} \in \partial f_k( y_{l,k} )
		,
		\qquad
		\xi_{l,k} - x_{l,k} \in \frac{l}{2} \partial\parens*{\norm{\cdot}_Y^2} ( y_{l,k} - y )
		.
	\end{equation*}
	Note that the second inclusion implies
	\begin{equation*}
		\norm{\xi_{l,k} - x_{l,k}}_X
		=
		l \norm{y_{l,k} - y}_Y
		\quad\text{and}\quad
		\dual{\xi_{l,k}-x_{l,k}}{y_{l,k} - y}_Y
		=
		l \norm{y_{l,k} - y}_Y^2
		.
	\end{equation*}
	Together with the first inclusion and the Fenchel--Young identity
	we get
	\begin{align*}
		f_k( y_{l,k} ) + f_k\conjugate( x_{l,k} )
		=
		\dual{x_{l,k}}{y_{l,k}}_Y
		&=
		\dual{x_{l,k}}{y}_Y
		+
		\dual{x_{l,k} - \xi_{l,k}}{y_{l,k} - y}_Y
		+
		\dual{\xi_{l,k}}{y_{l,k} - y}_Y
		\\&
		=
		\dual{x_{l,k}}{y}_Y
		-
		l \norm{ y_{l,k} - y }_Y^2
		+
		\dual{\xi_{l,k}}{y_{l,k} - y}_Y
		.
	\end{align*}
	Consequently,
	the lower bounds on $f_k$ and $f_k\conjugate$ yield
	\begin{align*}
		l \norm{ y_{l,k} - y }_Y^2
		&=
		\dual{x_{l,k}}{y}_Y
		+
		\dual{\xi_{l,k}}{y_{l,k} - y}_Y
		-
		f_k( y_{l,k} )
		-
		f_k\conjugate( x_{l,k} )
		\\&
		\le
		\norm{y}_Y \norm{ x_{l,k} }_X
		+
		\frac1k \norm{y_{l,k} - y}_Y
		+
		C( \norm{y_{l,k}}_Y + 1)
		+
		C( \norm{x_{l,k}}_X + 1)
		.
	\end{align*}
	Together with
	$\norm{x_{l,k}}_X \le l \norm{y_{l,k} - y}_Y + 1/k$,
	this shows that $\norm{y_{l,k}}_Y$ is uniformly bounded.
	Hence, for each fixed $l \in \N$,
	the sequence $\seq{x_{l,k}}_k$ is bounded.
	Given $l \in \N$, we first choose a subsequence which
	realizes
	\begin{equation*}
		\limsup_{k \to \infty}\parens*{
			f_k(y_{l,k})
			+
			l \norm{y_{l,k} - y}_Y^2
		}
	\end{equation*}
	and,
	owing to the separability or reflexivity of $Y$,
	this subsequence has a further subsequence
	such that $x_{l,k_n} \weaklystar x_l$ in $X$.
	Thus,
	\begin{align*}
		\limsup_{k \to \infty}\parens*{
			f_k(y_{l,k})
			+
			l \norm{y_{l,k} - y}_Y^2
		}
		&=
		\lim_{n \to \infty} \parens*{
			f_{k_n}(y_{l,k_n})
			+
			l \norm{y_{l,k_n} - y}_Y^2
		}
		\\&
		=
		\lim_{n \to \infty} \parens*{
			\dual{x_{l,k_n}}{y}_Y - f_{k_n}\conjugate(x_{l,k_n})
			+
			\dual{\xi_{l,k_n}}{y_{l,k_n} - y}_Y
		}
		.
	\end{align*}
	Since $y_{l,k}$ is bounded and $\norm{\xi_{l,k}}_X \le 1/k \to 0$,
	we can continue with
	\begin{equation*}
		\limsup_{k \to \infty}\parens*{
			f_k(y_{l,k})
			+
			l \norm{y_{l,k} - y}_Y^2
		}
		\le
		\dual{x_l}{y}_Y - f\conjugate(x_l)
		\le
		f(y)
		<
		\infty
		.
	\end{equation*}
	Owing to \cref{lem:technical_lemma_from_attouch},
	we find a sequence $\seq{l_k} \subset \N$ with $l_k \to \infty$
	as $k \to \infty$ and
	\begin{equation}
		\label{eq:diagonal_sequence_inequality_dual}
		\limsup_{k \to \infty}\parens*{
			f_k(y_{l_k,k})
			+
			l_k \norm{y_{l_k,k} - y}_Y^2
		}
		\le
		f(y)
		<
		\infty
		.
	\end{equation}
	We define $y_k := y_{l_k,k}$.
	Using the lower bound for $f_k$, we find
	\begin{align*}
		f_k(y_k)
		&\ge
		-C (\norm{ y_k }_Y + 1 )
		\ge
		-C (\norm{ y_k - y }_Y + \norm{y}_Y + 1 )
		\\&
		\ge
		-\frac{l_k}{2} \norm{ y_k - y }_Y^2
		-\frac{C^2}{2 l_k}
		-C (\norm{y}_Y + 1 )
		.
	\end{align*}
	Inserting this into \eqref{eq:diagonal_sequence_inequality_dual}
	and using $l_k \to \infty$ shows $\norm{y_k - y}_Y \to 0$.
	Similarly,
	\eqref{eq:diagonal_sequence_inequality_dual} implies
	\begin{equation*}
		\limsup_{k \to \infty}
		f_k(y_k)
		\le
		\limsup_{k \to \infty}\parens*{
			f_k(y_k)
			+
			l_k \norm{y_k - y}_Y^2
		}
		\le
		f(y)
		.
	\end{equation*}
	Hence, $\seq{y_k}$ is the desired recovery sequence.
\end{proof}

In the second variation of \cref{prop:strong_dual_recovery},
we prove the existence of a weak-$\star$ convergent recovery sequence
for the conjugate sequence.
We adapt the idea of the proof of \cref{prop:strong_dual_recovery}.
However, we now pick minimizers of
\begin{equation*}
	\tilde x
	\mapsto
	f_k\conjugate(\tilde x)
	+
	\frac{l}{2}
	\sum_{j = 1}^l \dual{\tilde x - x}{y^{(j)}}_Y^2
	,
\end{equation*}
where $\seq{y^{(j)}} \subset Y$ is dense.
In order to get existence of a minimizer
and a bound on the norm of those minimizers,
we need an additional assumption on the functions $f_k\conjugate$.
\begin{proposition}
	\label{prop:weakstar_dual_recovery}
	We assume that the following conditions are satisfied.
	\begin{enumerate}[label=(\arabic*)]
		\item
			\label{prop:weakstar_dual_recovery:1}
			The space $Y$ is separable.
		\item
			\label{prop:weakstar_dual_recovery:2}
			There exist $\seq{r_k} \subset [0,\infty)$, $c > 0$, $J \in \N$ and $y^{(j)} \in Y$, $j = 1,\ldots,J$,
			such that
			\begin{equation*}
				\frac{c}{2} \norm{x}_X^2
				\le
				f_k\conjugate(x)
				+
				\frac12 \sum_{j = 1}^J \dual{x}{y^{(j)}}_Y^2
				\qquad
				\forall x \in B^X_{r_k}(0)
			\end{equation*}
			holds for all $k \in \N$
			and
			$r_k \to \infty$.
		\item
			\label{prop:weakstar_dual_recovery:3}
			There exists $C \ge 0$ such that
			$f_k \ge - C (\norm{\cdot}_Y + 1)$
			holds for all $k \in \N$.
	\end{enumerate}
	Then, \eqref{BPfs} implies \eqref{RSfcw}.
\end{proposition}
We mention that this result follows from
(the proof of) \cite[Theorem~5.1]{CorreaPerezArosSantander2025}
applied to $f_k\conjugate$ replaced by
\begin{equation*}
	x
	\mapsto
	f_k\conjugate(x)
	+
	\frac12 \sum_{j = 1}^J \dual{x}{y^{(j)}}_Y^2
	.
\end{equation*}
However, due to our stronger assumptions,
it is possible to give an elementary and direct proof.
\begin{proof}
	Assumption \ref{prop:weakstar_dual_recovery:3}
	implies that the infimal value of
	$f_k + C \norm{\cdot}_Y$ is
	at least $-C$.
	Owing to Fenchel duality,
	the Fenchel dual
	$-f_k\conjugate - \delta_{B_C^X(0)}$
	has a maximizer
	with objective value also at least $-C$,
	see \cite[Section~2.5.4]{BonnansShapiro2000}.
	Here, $\delta_{B_C^X(0)}$ is the indicator function
	of the closed ball $B_C^X(0) \subset X$.
	Consequently,
	there exists $\hat x_k \in X$
	with $\norm{\hat x_k}_X \le C$
	and $f_k\conjugate(\hat x_k) \le C$.

	We start by extending the set
	$\set{y^{(1)}, \ldots, y^{(J)}}$
	to a countable dense subset $\set{y^{(j)} \given j \in \N} \subset Y$.
	In case $J = 1$, we simply redefine $J := 2$.
	Thus, assumption \ref{prop:weakstar_dual_recovery:2}
	is still satisfied and we have $J \ge 2$ in any case.
	In order to prove \eqref{RSfcw},
	we fix $x \in X$.
	Clearly, it is enough to consider the case $f\conjugate(x) < \infty$.
	Given $k \in \N$ and $l \ge J$,
	we consider the function
	$g_{l,k} \colon X \to (-\infty,\infty]$ via
	\begin{equation*}
		g_{l,k}(\tilde x)
		:=
		f_k\conjugate(\tilde x)
		+
		\frac{l}{2}
		\sum_{j = 1}^l \dual{\tilde x - x}{y^{(j)}}_Y^2
		.
	\end{equation*}
	We show that for all $l \ge J$ and $k$ large enough (depending on $l$)
	this function has a minimizer on the ball $B^X_{r_k}(0)$.
	To this end, let $\tilde x \in X$ with $\norm{\tilde x}_X = r_k$
	be given.
	Using the inequality
	$(a - b)^2 \ge \frac12 a^2 - b^2$ for $a,b \in \R$,
	we estimate
	(using $l \ge J \ge 2$)
	\begin{equation*}
		g_{l,k}(\tilde x)
		\ge
		f_k\conjugate(\tilde x)
		+
		\sum_{j = 1}^J \dual{\tilde x - x}{y^{(j)}}_Y^2
		\ge
		f_k\conjugate(\tilde x)
		+
		\frac12 \sum_{j = 1}^J \dual{\tilde x}{y^{(j)}}_Y^2
		-
		\sum_{j = 1}^J \dual{x}{y^{(j)}}_Y^2
		\ge
		\frac{c}{2} r_k^2 - \hat C
		,
	\end{equation*}
	where $\hat C := \sum_{j = 1}^J \dual{x}{y^{(j)}}_Y^2$ is a constant.
	On the other hand,
	\begin{equation*}
		g_{l,k}(\hat x_k)
		=
		f_k\conjugate(\hat x_k)
		+
		\frac{l}{2} \sum_{j = 1}^l \dual{\hat x_k - x}{y^{(j)}}_Y^2
		\le
		C
		+
		\frac{l}{2} \sum_{j = 1}^l \parens*{C + \norm{x}_X}^2 \norm{y^{(j)}}_Y^2
		=:
		\tilde C_l
		.
	\end{equation*}
	Consequently, for each $l \ge J$, there exists $K_l \in \N$
	such that for all $k \ge K_l$ we have
	$\norm{\hat x_k}_X \le C < r_k$
	and
	$g_{l,k}(\tilde x) > g_{l,k}(\hat x_k)$ for all $\tilde x \in X$
	with $\norm{\tilde x}_X = r_k$.
	Since $g_{l,k}$ is convex, this implies
	$g_{l,k}(\tilde x) > g_{l,k}(\hat x_k)$ for all $\tilde x \in X$
	with $\norm{\tilde x}_X \ge r_k$.
	Thus, we can use a standard argument to check that $g_{l,k}$
	possesses a minimizer lying in $B_{r_k}^X(0)$.
	For $k \ge K_l$,
	we select
	\begin{equation*}
		x_{l,k}
		\in
		\argmin_{\tilde x \in X}
		g_{l,k}(\tilde x)
		=
		\argmin_{\tilde x \in X}
		\parens*{
			f_k\conjugate(\tilde x)
			+
			\frac{l}{2}
			\sum_{j = 1}^l \dual{\tilde x - x}{y^{(j)}}_Y^2
		}
		\subset
		B^X_{r_k}(0)
		.
	\end{equation*}
	For completeness, we also set $x_{l,k} := x$
	for $l < J$ or $k < K_l$.

	Consider $l \ge J$ and $k \ge K_l$.
	We evaluate the optimality condition at $x_{l,k}$.
	Since the term under the sum is continuous, the chain rule applies.
	Consequently,
	\begin{equation*}
		y_{l,k}
		:=
		- l \sum_{j = 1}^l \dual{x_{l,k} - x}{y^{(j)}}_Y y^{(j)}
		\in
		\partial f_k\conjugate(x_{l,k})
		.
	\end{equation*}
	This directly implies the bound
	\begin{equation}
		\label{eq:bound_ymn}
		\norm{y_{l,k}}_Y
		\le
		l \sum_{j = 1}^l \abs{\dual{ x_{l,k} - x}{y^{(j)}}_Y} \norm{y^{(j)}}_Y
	\end{equation}
	and from the Fenchel--Young identity we get
	\begin{align*}
		f_k(y_{l,k}) + f_k\conjugate(x_{l,k})
		&=
		\dual{x_{l,k}}{y_{l,k}}_Y
		=
		\dual{x_{l,k} - x}{y_{l,k}}_Y
		+
		\dual{x}{y_{l,k}}_Y
		\\
		&=
		- l \sum_{j = 1}^l \dual{x_{l,k} - x}{y^{(j)}}_Y^2
		+
		\dual{x}{y_{l,k}}_Y
		,
	\end{align*}
	i.e.,
	\begin{equation}
		\label{eq:estimate_fnconj}
		f_k\conjugate(x_{l,k})
		+
		l \sum_{j = 1}^l \dual{x_{l,k} - x}{y^{(j)}}_Y^2
		=
		\dual{x}{y_{l,k}}_Y
		-
		f_k(y_{l,k})
		.
	\end{equation}
	The growth assumption on $f_k\conjugate$
	and $x_{l,k} \in B^X_{r_k}(0)$
	can be used to estimate
	\begin{align*}
		\frac{c}{2} \norm{ x_{l,k} }_X^2
		&\le
		f_k\conjugate(x_{l,k})
		+
		\frac12 \sum_{j = 1}^J \dual{x_{l,k}}{y^{(j)}}_Y^2
		\\
		&=
		f_k\conjugate(x_{l,k})
		+
		\frac12 \sum_{j = 1}^J \parens*{
			\dual{x_{l,k} - x}{y^{(j)}}_Y
			+
			\dual{x}{y^{(j)}}_Y
		}^2
		\\
		&\le
		f_k\conjugate(x_{l,k})
		+
		\sum_{j = 1}^J \parens*{
			\dual{x_{l,k} - x}{y^{(j)}}_Y^2
			+
			\dual{x}{y^{(j)}}_Y^2
		}
		.
	\end{align*}
	We again use
	$\hat C := \sum_{j = 1}^J \dual{x}{y^{(j)}}_Y^2$.
	Together with \eqref{eq:estimate_fnconj}, we get
	\begin{align}
		\label{eq:estimate_for_uniform_bound}
		\frac{c}{2} \norm{ x_{l,k} }_X^2
		+
		(l-1)
		\sum_{j = 1}^l \dual{x_{l,k} - x}{y^{(j)}}_Y^2
		&\le
		f_k\conjugate(x_{l,k})
		+
		l \sum_{j = 1}^l \dual{x_{l,k} - x}{y^{(j)}}_Y^2
		+
		\hat C
		\\
		\nonumber
		&=
		\dual{x}{y_{l,k}}_Y
		-
		f_k(y_{l,k})
		+
		\hat C
		.
	\end{align}
	Next, we use the lower bound on $f_k$ as well as \eqref{eq:bound_ymn},
	to arrive at
	\begin{align*}
		(l-1)
		\sum_{j = 1}^l \dual{x_{l,k} - x}{y^{(j)}}_Y^2
		&\le
		\norm{y_{l,k}}_Y
		(\norm{x}_X + C)
		+
		C + \hat C
		\\&\le
		l \sum_{j = 1}^l \abs{\dual{ x_{l,k} - x}{y^{(j)}}_Y} \norm{y^{(j)}}_Y
		(\norm{x}_X + C)
		+
		C + \hat C
		.
	\end{align*}
	This inequality holds for all $l \ge J$ and $k \ge K_l$.
	From this estimate we obtain that for fixed $l \ge J$,
	the sequences $\seq{\dual{x_{l,k} - x}{y^{(j)}}_Y}_k \subset \R$, $j = 1,\ldots,l$,
	are bounded.
	For fixed $l \ge J$, we first select a subsequence which realizes
	\begin{equation*}
		\limsup_{k \to \infty}
		\braces*{
			f_k\conjugate(x_{l,k})
			+
			l \sum_{j = 1}^l \dual{x_{l,k} - x}{y^{(j)}}_Y^2
		}
	\end{equation*}
	and, afterwards,
	we extract a subsequence such that
	$\dual{x_{l,k_n} - x}{y^{(j)}}_Y \to \alpha_{l,j} \in \R$
	as $n \to \infty$
	for all $j = 1,\ldots, l$.
	Note that this implies
	\begin{equation*}
		y_{l,k_n}
		=
		- l \sum_{j = 1}^l \dual{x_{l,k_n} - x}{y^{(j)}}_Y y^{(j)}
		\to
		- l \sum_{j = 1}^l \alpha_{l,j} y^{(j)}
		=:
		y_l
		\in
		Y
	\end{equation*}
	as $n \to \infty$
	strongly in $Y$
	and this allows us to use the assumption on $f_k$.
	Indeed, together with
	\eqref{eq:estimate_fnconj} and \cref{lem:aux_lower_bound}
	we obtain
	\begin{align*}
		\MoveEqLeft
		\limsup_{k \to \infty}
		\braces*{
			f_k\conjugate(x_{l,k})
			+
			l \sum_{j = 1}^l \dual{x_{l,k} - x}{y^{(j)}}_Y^2
		}
		\\
		&=
		\lim_{n \to \infty}
		\braces*{
			f_{k_n}\conjugate(x_{l,k_n})
			+
			l \sum_{j = 1}^l \dual{x_{l,k_n} - x}{y^{(j)}}_Y^2
		}
		\\&
		=
		\lim_{n \to \infty}
		\braces*{
			\dual{x}{y_{l,k_n}}_Y
			-
			f_{k_n}(y_{l,k_n})
		}
		=
		\dual{x}{y_l}_Y
		-
		\liminf_{n \to \infty}
		f_{k_n}(y_{l,k_n})
		\\&
		\le
		\dual{x}{y_l}_Y
		-
		f(y_l)
		.
	\end{align*}
	Using the definition of $f\conjugate$,
	the last expression can be bounded by $f\conjugate(x)$.
	Consequently,
	for each fixed $l \ge J$ we have
	\begin{equation*}
		\limsup_{k \to \infty}
		\braces*{
			f_k\conjugate(x_{l,k})
			+
			l \sum_{j = 1}^l \dual{x_{l,k} - x}{y^{(j)}}_Y^2
		}
		\le
		f\conjugate(x)
		<
		\infty
		.
	\end{equation*}
	This implies
	\begin{equation*}
		\limsup_{l \to \infty}
		\limsup_{k \to \infty}
		\braces*{
			f_k\conjugate(x_{l,k})
			+
			l \sum_{j = 1}^l \dual{x_{l,k} - x}{y^{(j)}}_Y^2
		}
		\le
		f\conjugate(x)
		.
	\end{equation*}
	Now, we apply \cref{lem:technical_lemma_from_attouch}
	and this gives us a sequence $\seq{l_k} \subset \N$
	with $l_k \to \infty$
	such that
	\begin{align*}
		\MoveEqLeft
		\limsup_{k \to \infty}
		\braces*{
			f_k\conjugate(x_{l_k,k})
			+
			l_k \sum_{j = 1}^{l_k} \dual{x_{l_k,k} - x}{y^{(j)}}_Y^2
		}
		\\&
		\le
		\limsup_{l \to \infty}
		\limsup_{k \to \infty}
		\braces*{
			f_k\conjugate(x_{l,k})
			+
			l \sum_{j = 1}^l \dual{x_{l,k} - x}{y^{(j)}}_Y^2
		}
		\le
		f\conjugate(x)
		<
		\infty
		.
	\end{align*}
	We set $x_k := x_{l_k,k}$.
	In particular, this shows
	$\limsup_{k \to \infty} f_k\conjugate(x_k) \le f\conjugate(x)$,
	and together with \eqref{eq:estimate_for_uniform_bound}
	we see that $\seq{x_k} \subset X$ is bounded
	and
	$\dual{x_k - x}{y^{(j)}}_Y \to 0$.
	The density of $\seq{y^{(j)}} \subset Y$
	implies
	$x_k \weaklystar x$
	and this finishes the proof.
\end{proof}

Finally,
we demonstrate that the assumption
$f_k \ge -C (\norm{\cdot}_Y + 1)$
is not restrictive.
\begin{lemma}
	\label{lem:uniformly_proper_bound}
	Assume that
	the sequence $\seq{f_k\conjugate}$ is uniformly proper,
	i.e.,
	there exists a bounded sequence $\seq{x_{0,k}} \subset X$
	such that $f_k\conjugate(x_{0,k})$ is bounded from above.
	Then, there exists a constant $C \ge 0$ such that
	$f_k \ge - C (\norm{\cdot}_Y + 1)$
	for all $k \in \N$.
\end{lemma}
\begin{proof}
	From the Fenchel--Young inequality, we find
	\begin{align*}
		f_k(y)
		\ge
		\dual{x_{0,k}}{y}_Y - f_k\conjugate(x_{0,k})
		&\ge
		-\norm{y}_Y \norm{x_{0,k}}_X - f_k\conjugate(x_{0,k})
		\\&
		\ge
		-C \parens*{ \norm{y}_Y + 1 }
		\quad\forall y \in Y, k \in \N
	\end{align*}
	for some appropriate constant $C \ge 0$.
\end{proof}
A similar argument can be used to get
$f_k\conjugate \ge -C (\norm{\cdot}_X + 1)$.
However,
we provide an alternative argument which is also required later on.
\begin{lemma}
	\label{lem:bound_on_fk_conjugate}
	Assume that $Y$ is separable or reflexive,
	$\seq{f_k\conjugate}$ is uniformly proper
	and that
	\eqref{BPfcw} is satisfied.
	Then, there exists $C \ge 0$
	such that
	$f_k\conjugate \ge -C (\norm{\cdot}_X + 1)$
	for all $k \in \N$.
\end{lemma}
\begin{proof}
	We argue by contradiction.
	Consequently, for every $n \in \N$, the condition
	$f_k\conjugate \ge -n (\norm{\cdot}_X + 1)$
	is violated for some $k \in \N$.
	However,
	since for every $k$, there exists $C_k \ge 0$
	with
	$f_k\conjugate \ge -C_k (\norm{\cdot}_X + 1)$
	(since $f_k\conjugate \in \Gamma_0(X)$ has a continuous and affine lower bound),
	the condition
	$f_k\conjugate \ge -n (\norm{\cdot}_X + 1)$
	must be violated for infinitely many $k \in \N$.
	Consequently, we find a strictly increasing sequence $\seq{k_n} \subset \N$
	such that
	$f_{k_n}\conjugate \ge -n (\norm{\cdot}_X + 1)$
	is violated for all $n \in \N$.
	In particular, we can choose
	$x_{k_n} \in X$ such that
	$f_{k_n}\conjugate(x_{k_n}) \le -n (\norm{x_{k_n}}_X + 1)$.

	Now, let $\seq{x_{0,k}} \subset X$ be a sequence with
	$\norm{x_{0,k}}_X \le C$ and $f_k\conjugate(x_{0,k}) \le C$
	for some $C \ge 0$.
	We define
	\begin{equation*}
		\lambda_n := \frac{1}{\norm{x_{k_n}}_X + 1} \in (0,1)
		\qquad\text{and}\qquad
		\bar x_n := (1-\lambda_n) x_{0,k_n} + \lambda_n x_{k_n}
		.
	\end{equation*}
	Consequently,
	\begin{align*}
		f_{k_n}\conjugate(\bar x_n)
		&\le
		(1-\lambda_n) f_{k_n}\conjugate( x_{0,k_n} )
		+
		\lambda_n f_{k_n}\conjugate( x_{k_n} )
		\le
		C - n
		\\
		\norm{ \bar x_n }_X
		&\le
		(1-\lambda_n) \norm{ x_{0,k_n} }_X
		+
		\lambda_n \norm{ x_{k_n} }_X
		\le
		C + 1
		.
	\end{align*}
	Since $Y$ is separable or reflexive,
	we can extract a subsequence $\bar x_{n_l} \weaklystar \bar x$.
	Together with \cref{lem:aux_lower_bound}, we infer
	\begin{equation*}
		f\conjugate(\bar x)
		\le
		\liminf_{l \to \infty} f_{k_{n_l}}\conjugate(\bar x_{n_l})
		\le
		\liminf_{l \to \infty} \braces*{ C - n_l }
		=
		-\infty
	\end{equation*}
	which contradicts $f\conjugate \in \Gamma_0(X)$.
\end{proof}
A similar argument can be used for the primal sequence $\seq{f_k}$.

\begin{lemma}
	\label{lem:bound_on_fk}
	Assume that
	\eqref{RSfs} and \eqref{BPfs} are satisfied.
	Then, there exists $C \ge 0$
	such that
	$f_k \ge -C (\norm{\cdot}_Y + 1)$
	for all $k \in \N$.
\end{lemma}
\begin{proof}
	We argue by contradiction.
	Consequently, for every $n \in \N$, the condition
	$f_k \ge -n (\norm{\cdot}_Y + 1)$
	is violated for some $k \in \N$.
	However,
	since for every $k$, there exists $C_k \ge 0$
	with
	$f_k \ge -C_k (\norm{\cdot}_Y + 1)$
	(since $f_k \in \Gamma_0(Y)$ has a continuous and affine lower bound),
	the condition
	$f_k \ge -n (\norm{\cdot}_Y + 1)$
	must be violated for infinitely many $k \in \N$.
	Consequently, we find a strictly increasing sequence $\seq{k_n} \subset \N$
	such that
	$f_{k_n} \ge -n (\norm{\cdot}_X + 1)$
	is violated for all $n \in \N$.
	In particular, we can choose
	$y_{k_n} \in Y$ such that
	$f_{k_n}(y_{k_n}) \le -n (\norm{y_{k_n}}_Y + 1)$.

	The properness of $f$ ensures the existence of $y_0 \in Y$
	with $f(y_0) < \infty$.
	Due to \eqref{RSfs}, we find a sequence
	$\seq{y_{0,k}} \subset Y$ with
	$y_{0,k} \to y_0$ and $f_k(y_{0,k}) \le C$
	for some $C \ge 0$.
	We define
	\begin{equation*}
		\lambda_n := \frac{1}{\sqrt{n} \parens*{ \norm{y_{k_n}}_Y + 1 }} \in (0,1)
		\qquad\text{and}\qquad
		\bar y_n := (1-\lambda_n) y_{0,k_n} + \lambda_n y_{k_n}
		.
	\end{equation*}
	Consequently,
	\begin{align*}
		f_{k_n}(\bar y_n)
		&\le
		(1-\lambda_n) f_{k_n}( y_{0,k_n} )
		+
		\lambda_n f_{k_n}( y_{k_n} )
		\le
		C
		-
		\frac{n \parens*{ \norm{y_{k_n}}_Y + 1 }}{\sqrt{n} \parens*{ \norm{y_{k_n}}_Y + 1 }}
		\to
		-\infty
		\\
		\norm{ \bar y_n - y_0 }_Y
		&\le
		(1-\lambda_n) \norm{ y_{0,k_n} - y_0 }_Y
		+
		\lambda_n \norm{ y_{k_n} - y_0 }_Y
		\\&
		\le
		\norm{ y_{0,k_n} - y_0 }_Y
		+
		\frac{\norm{y_{k_n}}_Y + \norm{y_{0,k_n}}_Y}{\sqrt{n} \parens*{ \norm{y_{k_n}}_Y + 1 }}
		\to
		0
		.
	\end{align*}
	Consequently, $\bar y_n \to y_0$.
	Together with \cref{lem:aux_lower_bound}, we infer
	\begin{equation*}
		-\infty
		<
		f(y_0)
		\le
		\liminf_{l \to \infty} f_{k_{n_l}}(\bar y_{n_l})
		\le
		\liminf_{l \to \infty} \braces*{ C - \sqrt{n_l} }
		=
		-\infty
		,
	\end{equation*}
	which is a contradiction.
\end{proof}

\section{Graphical convergence of the subdifferentials}
\label{sec:subdifferentials}

We want to study whether the convergence properties
investigated in the last section
are related to convergence properties of the subdifferentials $\partial f_k$.
Similar to the previous section,
the results in this section
can already be found in \cite[Theorem~5.7]{CorreaPerezArosSantander2025},
which shows that
(under mild conditions)
the graphical convergence of the subdifferentials $\partial f_k$
(together with a normalization condition)
is equivalent to
the $\Gamma$-convergence of $f_k$
and to
the $\Gamma$-convergence of $f_k\conjugate$
(everything in appropriate topologies).

In what follows,
we review these implications
and give simple proofs.
In particular, our proof of \cref{prop:RSpf_implies_all}
only relies on the Rockafellar integration formula,
while
the corresponding proof in
\cite[Section~10.2]{CorreaPerezArosSantander2025}
introduces the Fenchel subdifferential
and uses associated integration results from
\cite{CombariThibault1998}.

For the convergence of the subdifferentials, we investigate
the following two conditions.
In these conditions,
we abuse the notation and write
``$\seq{(y_k, x_k)} \subset \seq{\graph\partial f_k}$''
to denote a sequence of pairs $(y_k, x_k) \in Y \times X$
satisfying $(y_k, x_k) \in \graph \partial f_k$ for all $k \in \N$.
The first condition
\begin{equation*}
	\tag{\BPpf}
	\label{BPpf}
	\begin{aligned}
		&
		\forall \seq{(y_k, x_k)} \subset \seq{\graph\partial f_k} :
		\forall (y,x) \in Y \times X :
		\\ &
		\qquad
		\parens*{
			y_k \to y
			\; \land \;
			x_k \weaklystar x
		}
		\quad\Rightarrow\quad
		(y,x) \in \graph \partial f
	\end{aligned}
\end{equation*}
says that limit points of subdifferentials
belong to the subdifferential of $f$,
whereas the second condition
\begin{equation*}
	\tag{\RSpf}
	\label{RSpf}
	\forall (y,x) \in \graph\partial f :
	\exists \seq{(y_k, x_k)} \subset \seq{\graph\partial f_k} :
	y_k \to y
	\; \land \;
	x_k \weaklystar x
\end{equation*}
provides a recovery sequence.
Further, we need the normalization condition
\begin{equation*}
	\tag{\NC}
	\label{NC}
	\begin{aligned}
		&
		\exists \seq{ (y_{0,k}, x_{0,k}) } \subset \seq{\graph\partial f_k} :
		\exists (y_0, x_0) \in \graph\partial f :
		\\ &
		\qquad
		y_{0,k} \to y_0
		\; \land \;
		x_{0,k} \weaklystar x_0
		\; \land \;
		f_k(y_{0,k}) \to f(y_0)
	\end{aligned}
\end{equation*}
since the subdifferential determines a convex function
only up to a constant, see \cref{thm:rockafellar_integration}.

Throughout this section, we utilize the same standing assumption as
in the previous section.
\begin{assumption}[Standing assumption]
	\label{asm:standing_subdifferentials}
	We have
	$\seq{f_k} \subset \Gamma_0(Y)$ and $f \in \Gamma_0(Y)$.
\end{assumption}

The proof of the first implication is very simple.
\begin{lemma}
	\label{lem:BP_subdifferential}
	Let \eqref{BPfs} and \eqref{BPfcw} be satisfied.
	Then, \eqref{BPpf} holds.
\end{lemma}
\begin{proof}
	Let a sequence $\seq{(y_k, x_k)}$ as in \eqref{BPpf} be given.
	The assumptions together with the Fenchel--Young identity
	imply
	\begin{align*}
		f(y) + f\conjugate(x)
		&\le
		\liminf_{k \to \infty} f_k(y_k) + \liminf_{k \to \infty} f_k\conjugate(x_k)
		\\
		&\le
		\liminf_{k \to \infty} \parens*{ f_k(y_k) + f_k\conjugate(x_k) }
		=
		\liminf_{k \to \infty} \parens*{ \dual{x_k}{y_k}_Y }
		=
		\dual{x}{y}_Y
		.
	\end{align*}
	Owing again to Fenchel--Young, this implies $(y,x) \in \graph\partial f$.
\end{proof}

The existence of a recovery sequence as in \eqref{RSpf}
is shown in \cite[Section~10.1]{CorreaPerezArosSantander2025}.
For convenience, we restate the proof.
\begin{proposition}
	\label{prop:RS_subdifferential}
	Suppose that \eqref{RSfs} and \eqref{RSfcw} are satisfied.
	Then, \eqref{RSpf} and \eqref{NC} hold.
\end{proposition}
\begin{proof}
	Let $(y,x) \in \graph\partial f$ be given, i.e.,
	\begin{equation*}
		f(y) + f\conjugate(x) = \dual{x}{y}_Y.
	\end{equation*}
	By the assumptions, we infer the existence of sequences
	$\seq{\tilde y_k} \subset Y$ and $\seq{\tilde x_k} \subset X$
	with
	$\tilde y_k \to y$, $\tilde x_k \weaklystar x$,
	$\limsup_{k \to \infty} f_k(\tilde y_k) \le f(y)$,
	and
	$\limsup_{k \to \infty} f_k\conjugate(\tilde x_k) \le f\conjugate(x)$.
	We define
	\begin{equation*}
		\varepsilon_k := f_k(\tilde y_k) + f_k\conjugate(\tilde x_k) - \dual{\tilde x_k}{\tilde y_k}_Y \ge 0
	\end{equation*}
	and observe
	\begin{align*}
		\limsup_{k \to \infty} \varepsilon_k
		&\le
		\limsup_{k \to \infty} f_k(\tilde y_k)
		+
		\limsup_{k \to \infty} f_k\conjugate(\tilde x_k)
		-
		\lim_{k \to \infty} \dual{\tilde x_k}{\tilde y_k}_Y
		\\&
		\le
		f(y) + f\conjugate(x) - \dual{x}{y}_Y
		=
		0,
	\end{align*}
	i.e., $\varepsilon_k \to 0$.
	Now, (the proof of) \cite[Lemma]{BrondstedRockafellar1965}
	provides us with
	$(y_k, x_k) \in \graph \partial f_k$
	such that
	$\norm{ y_k - \tilde y_k }_Y \le \sqrt{\varepsilon_k}$,
	$\norm{ x_k - \tilde x_k }_X \le \sqrt{\varepsilon_k}$,
	as well as
	$f_k(y_k) \le f_k(\tilde y_k) + \dual{\tilde x_k}{y_k - \tilde y_k}_Y$.
	Consequently,
	$y_k \to y$ and $x_k \weaklystar x$.
	This finishes the proof of \eqref{RSpf}.

	In order to show \eqref{NC},
	we apply the previous argument to an arbitrary
	$(y_0, x_0) \in \graph\partial f$ (which exists)
	to obtain
	a sequence
	$\seq{(y_{0,k}, x_{0,k})} \subset \seq{\graph\partial f_k}$
	such that
	$y_{0,k} \to y_0$ in $Y$,
	$x_{0,k} \weaklystar x_0$ in $X$.
	Additionally, we obtain
	\begin{equation*}
		f_k(y_{0,k})
		\le
		f_k(\tilde y_{0,k}) + \dual{\tilde x_{0,k}}{y_{0,k} - \tilde y_{0,k}}_Y
		,
	\end{equation*}
	where
	$\norm{\tilde y_{0,k} - y_{0,k}}_Y \le \sqrt{\varepsilon_k}$,
	$\norm{\tilde x_{0,k} - x_{0,k}}_X \le \sqrt{\varepsilon_k}$,
	and
	$\limsup_{k \to \infty} f_k(\tilde y_{0,k}) \le f(y_0)$.
	Consequently,
	\begin{equation*}
		\limsup_{k \to \infty}
		f_k(y_{0,k})
		\le
		\limsup_{k \to \infty}
		f_k(\tilde y_{0,k})
		+
		\lim_{k \to \infty}
		\dual{\tilde x_{0,k}}{y_{0,k} - \tilde y_{0,k}}_Y
		\le
		f(y_0) + 0
		.
	\end{equation*}
	Finally,
	since \eqref{RSfcw} implies \eqref{BPfs}, see \cref{lem:lower_bounds},
	we also have
	$f(y_0) \le \liminf_{k \to \infty} f_k(y_{0,k})$.
	Consequently, $f_k(y_{0,k}) \to f(y_0)$
	and this finishes the proof of \eqref{NC}.
\end{proof}

Next, we address the converse implications.
In order to obtain convergence properties of $f_k$
from those of $\partial f_k$,
the Rockafellar integration formula from \cref{thm:rockafellar_integration}
is crucial.

\begin{proposition}
	\label{prop:RSpf_implies_all}
	Assume that \eqref{RSpf} and \eqref{NC} hold.
	Then, \eqref{BPfs} and \eqref{RSfs} are satisfied.
\end{proposition}
\begin{proof}
	In order to prove \eqref{BPfs},
	let $\bar y \in Y$ and $\seq{\bar y_k} \subset Y$ with $\bar y_k \to \bar y$ be given.
	For an arbitrary $M \in \R$ satisfying $M < f(\bar y)$,
	we can use \cref{thm:rockafellar_integration}
	to get the existence of
	$n \in \N$ and $(y_i, x_i) \in \graph\partial f$, $i = 1,\ldots,n$
	such that
	\begin{equation*}
		M
		\le
		f(y_0)
		+ 
		\sum_{i = 0}^{n} \dual{x_i}{y_{i+1} - y_i}_Y
		=:
		\tilde M
		,
	\end{equation*}
	where
	$y_{n+1} := \bar y$
	and
	$(y_0, x_0) \in \graph\partial f$
	is given by \eqref{NC}.
	According to \eqref{RSpf},
	for each $i = 1,\ldots,n$ we find a sequence
	$\seq{(y_{i,k}, x_{i,k})}_k \subset \seq{\graph\partial f_k}$
	such that
	$y_{i,k} \to y_i$ and $x_{i,k} \weaklystar x_i$.
	For $(y_0, x_0)$, we use the sequence given by \eqref{NC}.
	Using $y_{n+1,k} := \bar y_k$
	we get
	\begin{align*}
		M
		&\le
		\tilde M
		=
		\lim_{k \to \infty} \braces*{
			f(y_{0,k})
			+ 
			\sum_{i = 0}^{n} \dual{x_{i,k}}{y_{i+1,k} - y_{i,k}}_Y
		}
		\\&
		\le
		\liminf_{k \to \infty} f_k(\bar y_k)
		,
	\end{align*}
	where the last inequality uses again \cref{thm:rockafellar_integration}.
	Since $M < f(\bar y)$ was arbitrary, this implies the desired
	\begin{equation*}
		f(\bar y)
		\le
		\liminf_{k \to \infty} f_k(\bar y_k)
		.
	\end{equation*}

	Next, we prove \eqref{RSfs}.
	First, we check that we can find a recovery sequence for every
	$\bar y \in \dom\partial f$.
	We fix $\bar x \in \partial f(\bar y)$
	and denote by $\seq{(\bar y_k, \bar x_k)} \subset \seq{\graph\partial f_k}$
	the recovery sequence given by \eqref{RSpf}.
	Next, we apply again \cref{thm:rockafellar_integration}.
	This shows that for every $M \in \R$ with $f(\bar y) < M$
	we find $n \in \N$ and $(y_i, x_i) \in \graph\partial f$, $i = 1,\ldots,n$,
	such that
	\begin{equation*}
		f(y_0)
		\le
		M
		+
		\dual{\bar x}{y_1 - \bar y}_Y
		+
		\sum_{i = 1}^{n} \dual{x_i}{y_{i+1} - y_i}_Y
		,
	\end{equation*}
	where $y_{n+1} := y_0$.
	As in the first part of the proof, we apply \eqref{RSpf} and \eqref{NC}
	and set $y_{n+1,k} := y_{0,k}$
	to get
	\begin{align*}
		\lim_{k \to \infty} f_k(y_{0,k})
		&
		=
		f(y_0)
		\\&
		\le
		M
		+
		\lim_{k \to \infty} \braces*{
			\dual{\bar x_k}{y_{1,k} - \bar y_k}_Y
			+
			\sum_{i = 1}^{n} \dual{x_{i,k}}{y_{i+1,k} - y_{i,k}}_Y
		}
		.
	\end{align*}
	From \cref{thm:rockafellar_integration}, we find that
	$f_k(y_{0,k}) - f_k(\bar y_k)$
	is an upper bound for the term inside the braces.
	Consequently,
	\begin{equation*}
		\lim_{k \to \infty} f_k(y_{0,k})
		\le
		M
		+
		\lim_{k \to \infty} f_k(y_{0,k})
		+
		\liminf_{k \to \infty} \braces*{ -f_k(\bar y_k) }
		.
	\end{equation*}
	Since $M > f(\bar y)$ was arbitrary, this shows the desired
	\begin{equation*}
		\limsup_{k \to \infty} f_k(\bar y_k)
		\le
		f(\bar y).
	\end{equation*}
	For later reference, we mention that
	together with the first part of the proof
	we even get
	$f(\bar y) = \lim_{k \to \infty} f_k(\bar y_k)$.

	Now, let $\bar y \in Y$ we arbitrary.
	Of course, it is sufficient to consider the case $f(\bar y) < \infty$.
	Owing to \cite[Theorem~2]{BrondstedRockafellar1965},
	we find a sequence $\seq{\tilde y_l} \subset \dom\partial f$
	such that $\tilde y_l \to \bar y$ and $f(\tilde y_l) \to f(\bar y)$.
	According to the previous argument,
	we find sequences $\seq{\bar y_{l,k}}_k$ such that
	$\bar y_{l,k} \to \tilde y_l$ as $k \to \infty$
	and
	$\lim_{k \to \infty} f_k(\bar y_{l,k}) = f(\bar y_l)$.
	Consequently,
	\begin{equation*}
		\limsup_{l \to \infty}
		\limsup_{k \to \infty}
		\braces*{
			\abs{ f_k(\bar y_{l,k}) - f(\bar y) }
			+
			\norm{\bar y_{l,k} - \bar y}_Y
		}
		\le
		\lim_{l \to \infty}
		\braces*{
			\abs{ f(\bar y_l) - f(\bar y) }
			+
			\norm{\tilde y_l - \bar y}_Y
		}
		=
		0
		.
	\end{equation*}
	Now, \cref{lem:technical_lemma_from_attouch} yields $\seq{l_k} \subset \N$ with $l_k \to \infty$
	such that $\bar y_k := \bar y_{l_k, k}$ satisfies
	\begin{equation*}
		\limsup_{k \to \infty}
		\braces*{
			\abs{ f_k(\bar y_k) - f(\bar y) }
			+
			\norm{\bar y_k - \bar y}_Y
		}
		\le
		0
		.
	\end{equation*}
	In particular, this shows $\bar y_k \to \bar y$
	and $f_k(\bar y_k) \to f(\bar y)$.
\end{proof}
At this point,
we would like to mention that
the first part of the proof
can be modified to conclude
\eqref{BPfcw} directly,
but this also follows from the previous result together with \cref{lem:lower_bounds}.
When doing so, however, there is a crucial subtlety in nonreflexive spaces:
One needs an appropriate variant of \cref{thm:rockafellar_integration}
for conjugate functions which only involves $\partial f$
and not $\partial f\conjugate$,
since $\graph \partial f\conjugate \subset X \times X\dualspace = Y\dualspace \times Y\bidualspace$
might be bigger than
$\graph \partial f \subset Y \times Y\dualspace$.
We briefly mention that such a result can be obtained by
combining \cref{thm:rockafellar_integration} with
\cite[Proposition~3.4]{BartzBauschkeBorweinReichWang2007}.

Similarly, the first half of the second part of the proof
can be used to check that for every $\bar x \in X$
with $\partial f\conjugate(\bar x) \cap Y \ne \emptyset$,
there exists a sequence $\seq{\bar x_k}$
with $\bar x_k \weaklystar \bar x$
and
$f\conjugate(\bar x) = \lim_{k \to \infty} f_k\conjugate(\bar x_k)$.
However,
the diagonalization argument at the very end of the proof
cannot be transferred, since this would require a uniform bound
on $\seq{\norm{\bar x_{l,k}}_X}$ which cannot be obtained
without any additional assumption.

By collecting the obtained results,
we obtain the following theorem which is inspired by and very similar to
\cite[Theorem~5.7]{CorreaPerezArosSantander2025}.
For convenience, we repeat the standing assumptions in the formulation.
\begin{theorem}
	\label{thm:main_result_conjugate}
	Assume that the Banach space $Y$ is reflexive or separable.
	Further, let $\seq{f_k} \subset \Gamma_0(Y)$ and $f \in \Gamma_0(Y)$ be given.
	Consider the following statements.
	\begin{enumerate}
		\item
			\label{thm:main_result_conjugate:1}
			\eqref{BPfcw} and \eqref{RSfcw} hold.
		\item
			\label{thm:main_result_conjugate:2b}
			\eqref{RSpf} and \eqref{NC} hold.
		\item
			\label{thm:main_result_conjugate:2}
			\eqref{BPpf}, \eqref{RSpf} and \eqref{NC} hold.
		\item
			\label{thm:main_result_conjugate:3}
			\eqref{BPfs} and \eqref{RSfs} hold.
	\end{enumerate}
	Then,
	\ref{thm:main_result_conjugate:1}
	$\Rightarrow$
	\ref{thm:main_result_conjugate:2b}
	$\Leftrightarrow$
	\ref{thm:main_result_conjugate:2}
	$\Rightarrow$
	\ref{thm:main_result_conjugate:3}.
	Moreover, if $Y$ is separable and
	\itemref{prop:weakstar_dual_recovery:2}
	holds,
	then
	\ref{thm:main_result_conjugate:3}
	$\Rightarrow$
	\ref{thm:main_result_conjugate:1}.
\end{theorem}
\begin{proof}
	``\ref{thm:main_result_conjugate:1} $\Rightarrow$ \ref{thm:main_result_conjugate:2}'':
	Clearly, \eqref{RSfcw} implies that $\seq{f_k\conjugate}$ is uniformly proper.
	Consequently, \cref{lem:uniformly_proper_bound,lem:bound_on_fk_conjugate}
	imply the existence of $C \ge 0$ with
	$f_k \ge -C(\norm{\cdot}_Y + 1)$
	and
	$f_k\conjugate \ge -C(\norm{\cdot}_X + 1)$.
	Thus, we can invoke \cref{prop:strong_primal_recovery} to get \eqref{RSfs}.
	Moreover, \cref{lem:lower_bounds} provides us with \eqref{BPfs}.
	Thus, we can use
	\cref{lem:BP_subdifferential} and \cref{prop:RS_subdifferential}
	to deduce
	\ref{thm:main_result_conjugate:2}.

	``\ref{thm:main_result_conjugate:2} $\Rightarrow$ \ref{thm:main_result_conjugate:2b}'':
	This is clear.

	``\ref{thm:main_result_conjugate:2b} $\Rightarrow$ \ref{thm:main_result_conjugate:3}'':
	This is precisely \cref{prop:RSpf_implies_all}.

	``\ref{thm:main_result_conjugate:2b} $\Rightarrow$ \ref{thm:main_result_conjugate:2}'':
	Since 
	\ref{thm:main_result_conjugate:2b} implies \ref{thm:main_result_conjugate:3},
	we can invoke 
	\itemref{lem:lower_bounds:1} to derive \eqref{BPfcw}.
	Thus, we can invoke \cref{lem:BP_subdifferential}
	to deduce \eqref{BPpf}.
	This shows \ref{thm:main_result_conjugate:2}.

	``\ref{thm:main_result_conjugate:3} $\Rightarrow$ \ref{thm:main_result_conjugate:1}'':
	From \cref{lem:lower_bounds} we deduce \eqref{BPfcw}.
	In order to invoke \cref{prop:weakstar_dual_recovery},
	we note that \cref{lem:bound_on_fk}
	provides $C \ge 0$ such that
	$f_k \ge -C(\norm{\cdot}_Y + 1)$.
	Consequently, all assumptions of \cref{prop:weakstar_dual_recovery}
	are satisfied (under the additional assumptions)
	and this proves \eqref{RSfcw}.
\end{proof}

\section{Second-order conditions via subderivatives and the NDC}
\label{sec:SSC}

In this section,
we review the derivation of second-order conditions
for \eqref{eq:problem}
via weak-$\star$ second subderivatives.
Special attention is paid to
the non-degeneracy condition \eqref{eq:NDC}.
In particular,
we prove a novel characterization of this condition.

\begin{assumption}[Standing assumption]
	\label{asm:standing_subderivatives}
	We assume that $Y$ is a separable Banach space and we recall $X = Y\dualspace$.
	Further, we are given functions $G \colon X \to (-\infty,\infty]$
	and $F \colon \dom G \to \R$.
	We fix $\bar x \in \dom G$.

	We further assume that, associated with $F$ and $\bar x$,
	there exist $F'(\bar x) \in Y$
	and a bounded bilinear form $F''(\bar x) \colon X \times X \to \R$
	such that $X \ni h \mapsto F''(\bar x) h^2$
	is sequentially weak-$\star$ lower semicontinuous
	and such that
	\begin{equation}
		\label{eq:Taylor_F}
		\lim_{k \to \infty}
		\frac{
			F(\bar x + t_k h_k) - F(\bar x) - t_k \dual{h_k}{F'(\bar x)}_Y - \frac12 t_k^2 F''(\bar x) h_k^2
		}{ t_k^2 }
		=
		0
	\end{equation}
	holds for all sequences $\seq{t_k} \subset (0,\infty)$,
	$\seq{h_k} \subset X$ such that
	$t_k \to 0$, $h_k \weaklystar h$ and
	$\bar x + t_k h_k \in \dom G$ for all $k \in \N$.

	For convenience, we define $\bar\xi := -F'(\bar x) \in Y$.
\end{assumption}
Note that we use the standard abbreviation
$F''(\bar x) h_k^2 := F''(\bar x)[ h_k, h_k ]$.

The assumption \eqref{eq:Taylor_F} is similar to a second-order Taylor expansion.
However, due to the additional assumption $\bar x + t_k h_k \in \dom G$,
it can also be satisfied in situations in which $F$ is not twice differentiable on $X$,
see, e.g.,
\cite[Section~6.2]{ChristofWachsmuth2017:1}.

For the function $G$, we do not assume a Taylor expansion as in \eqref{eq:Taylor_F}.
Instead,
the theory is based on the notion of the weak-$\star$ second subderivative.
\begin{definition}
	\label{def:subderivative}
	The weak-$\star$ second subderivative
	$\sd^2 G(\bar x, \bar\xi; \cdot) \colon X \to [-\infty,\infty]$
	of $G$ at $\bar x$ for $\bar\xi$
	is defined via
	\begin{equation*}
		\sd^2 G(\bar x, \bar\xi; h)
		=
		\inf\set*{
			\liminf_{k \to \infty}
			\frac{ G(\bar x + t_k h_k) - G(\bar x) - t_k \dual{h_k}{\bar\xi}_Y }{\frac12 t_k^2}
			\given
			t_k \searrow 0, \; h_k \weaklystar h
		}
	\end{equation*}
	for all $h \in X$.
	Here, the infimum is taken w.r.t.\ all sequences $\seq{t_k} \subset (0,\infty)$,
	$\seq{h_k} \subset X$ such that
	$t_k \to 0$ and $h_k \weaklystar h$.

	We further say that $G$ is
	weak-$\star$ twice epidifferentiable
	at $\bar x$ for $\bar\xi$
	if for every $h \in X$
	and every sequence $\seq{t_k} \subset (0,\infty)$ with $t_k \to 0$
	there exists a sequence
	$\seq{h_k} \subset X$
	such that
	\begin{equation*}
		\sd^2 G(\bar x, \bar\xi; h)
		=
		\lim_{k \to \infty}
		\frac{ G(\bar x + t_k h_k) - G(\bar x) - t_k \dual{h_k}{\bar\xi}_Y }{\frac12 t_k^2}
	\end{equation*}
	and
	$h_k \weaklystar h$.
	If the recovery sequences $\seq{h_k}$ can even be chosen such that $h_k \to h$,
	we say that $G$ is 
	strongly twice epidifferentiable at $\bar x$ for $\bar\xi$.
\end{definition}

At this point, we would like to mention that
``strong twice epidifferentiability'' in the sense of \cref{def:subderivative}
is precisely what is called
``twice epidifferentiability'' in \cite{Do1992},
``second order Mosco differentiability'' in \cite{Noll1995}
and
``twice Mosco epidifferentiability'' in \cite{Levy1999}.

Another crucial role in the theory is played
by the so-called non-degeneracy condition.
\begin{definition}
	\label{def:NDC}
	We say that the non-degeneracy condition is satisfied
	if
	\begin{equation*}
		\label{eq:NDC}
		\tag{NDC}
		\begin{aligned}
			&\text{for all $\seq{t_k} \subset (0,\infty)$, $\seq{h_k} \subset X$
				with $t_k \to 0$, $h_k \weaklystar 0$
			and $\norm{h_k}_X = 1$, we have}\\
			&\qquad
			\liminf_{k \to \infty} \parens[\bigg]{
				\frac{
					G(\bar x + t_k h_k)
					-
					G(\bar x)
					-
					t_k \dual{h_k}{\bar\xi}_Y
				}{t_k^2}
				+
				\frac12 F''(\bar x) h_k^2
			}
			> 0
			.
		\end{aligned}
	\end{equation*}
\end{definition}

The importance of this condition
is explained by the following, fundamental result.
For a proof, we refer to
\cite[Theorem~2.20]{BorchardWachsmuth2023}
and the references therein.

\begin{theorem}
	\label{thm:no-gap}
	In addition to the standing assumption, let one of the following conditions be satisfied.
	\begin{enumerate}[label=(\arabic*)]
		\item
			The map $X \ni h \mapsto F''(\bar x) h^2$ is sequentially weak-$\star$ upper semicontinuous
			(i.e., it is sequentially weak-$\star$ continuous).
		\item
			The functional $G$ is strongly twice epidifferentiable at $\bar x$ for $\bar\xi$.
	\end{enumerate}
	Then,
	the quadratic growth condition
	\begin{equation*}
		\tag{\textup{QG}}
		\label{eq:QG}
		(F + G)(x)
		\ge
		(F + G)(\bar x)
		+
		\frac{c}{2} \norm{x - \bar x}_X^2
		\qquad
		\forall x \in B^X_r(\bar x)
	\end{equation*}
	holds
	with constants $c, r > 0$
	if and only if
	the second-order condition
	\begin{equation*}
		\label{eq:SSC}
		\tag{SSC}
		\sd^2 G(\bar x, \bar\xi; h) + F''(\bar x) h^2
		>
		0
		\qquad\forall h \in X \setminus \set{0}
	\end{equation*}
	and
	the non-degeneracy condition \eqref{eq:NDC}
	are satisfied.
\end{theorem}

This result shows that the quadratic growth condition
cannot hold if \eqref{eq:NDC} is violated.
We further point out that the additional assumptions in
\cref{thm:no-gap} are only needed for the implication
``$\eqref{eq:QG}\Rightarrow\eqref{eq:SSC}$'',
whereas the other implications
``$\eqref{eq:QG}\Rightarrow\eqref{eq:NDC}$''
and
``$\eqref{eq:SSC} \land \eqref{eq:NDC}\Rightarrow\eqref{eq:QG}$''
hold without any further assumptions,
see \cite[Section~2]{BorchardWachsmuth2023}.

In the next result, we provide a novel characterization of \eqref{eq:NDC}.
\begin{theorem}
	\label{thm:characterization_of_NDC}
	In addition to \cref{asm:standing_subderivatives} we assume that
	$\sd^2 G( \bar x, \bar\xi; h) > -\infty$ for all $h \in X \setminus \set{0}$.
	Then the condition \eqref{eq:NDC} is satisfied
	if and only if
	there exist $r \in (0,\infty]$, $c > 0$, $n \in \N$ and $\xi_1, \ldots, \xi_n \in Y$ such that
	\begin{equation}
		\label{eq:growth_of_G}
		G(x)
		-
		G(\bar x)
		-
		\dual{x - \bar x}{\bar\xi}_Y
		+
		\frac12 F''(\bar x)(x - \bar x)^2
		\ge
		\frac{c}{2} \norm{x - \bar x}_X^2
		-
		\frac12 \sum_{j = 1}^n \dual{x - \bar x}{\xi_j}_Y^2
		\quad
		\forall x \in B^X_r(\bar x)
	\end{equation}
	holds.
\end{theorem}
\begin{proof}
	``$\Leftarrow$'':
	Let sequences $\seq{t_k} \subset (0,\infty)$ and $\seq{h_k} \subset X$
	as in \eqref{eq:NDC} be given.
	Since $t_k \to 0$ and since $\seq{h_k}$ is bounded, we have
	$\bar x + t_k h_k \in B^X_r(\bar x)$
	for $k$ large enough.
	Consequently, \eqref{eq:growth_of_G}
	can be applied with $x = \bar x + t_k h_k$ for $k$ large enough
	and this directly implies
	\begin{align*}
		\MoveEqLeft
		\liminf_{k \to \infty} \parens[\bigg]{
			\frac{
				G(\bar x + t_k h_k)
				-
				G(\bar x)
				-
				t_k \dual{h_k}{\bar\xi}_Y
			}{t_k^2}
			+
			\frac12 F''(\bar x) h_k^2
		}
		\\&
		\ge
		\liminf_{k \to \infty}
		\braces[\big]{
			\frac{c}{2} \norm{h_k}_X^2
			-
			\frac12 \sum_{j = 1}^n \dual{h_k}{\xi_j}_Y^2
		}
		\ge
		\frac{c}{2}
		,
	\end{align*}
	where we used $\norm{h_k}_X = 1$ and $h_k \weaklystar 0$ in the last estimate.

	``$\Rightarrow$'':
	Let $\set{\xi_j \given j \in \N} \subset Y$ be a dense subset of $Y$.
	We show the existence of $k \in \N$ such that
	\eqref{eq:growth_of_G} holds with $n = k$ and $r = 1/k$.
	We follow the idea of the proof of \cite[Theorem~2.9]{WachsmuthWachsmuth2021}.
	We abbreviate
	\begin{equation*}
		\Phi_k(x)
		:=
		G(x)
		-
		G(\bar x)
		-
		\dual{x - \bar x}{\bar\xi}_Y
		+
		\frac12 F''(\bar x)(x - \bar x)^2
		+
		\frac12 \sum_{j = 1}^k \dual{x - \bar x}{\xi_j}_Y^2
	\end{equation*}
	and
	\begin{equation*}
		c_k
		:=
		\inf\set*{
			\frac{\Phi_k(x)}{\frac12 \norm{x - \bar x}_X^2}
			\given
			x \in B^X_{1/k}(\bar x) \setminus \set{\bar x}
		}
		\in
		[-\infty, \infty]
		.
	\end{equation*}
	Note that $c_k > 0$ is equivalent to \eqref{eq:growth_of_G} being satisfied
	with $c = c_k$, $r = 1/k$ and $n = k$.
	Further, the sequence $c_k$ is increasing and, consequently, it is enough to show that
	\begin{equation*}
		\hat c
		:=
		\lim_{k \to \infty} c_k
		>
		0
		.
	\end{equation*}
	Owing to the definition of $c_k$, we can find a sequence $\seq{x_k} \subset X$
	such that $x_k \in B_{1/k}(\bar x) \setminus \set{\bar x}$ and
	\begin{equation*}
		\hat c
		=
		\lim_{k \to \infty} \frac{ \Phi_k(x_k)}{ \frac12 \norm{x_k - \bar x}_X^2 }
		.
	\end{equation*}
	We define
	$t_k := \norm{x_k - \bar x}_X > 0$,
	$h_k := (x_k - \bar x) / t_k$.
	Then, $t_k \to 0$ and $\norm{h_k}_X = 1$.
	Since $Y$ is assumed to be separable, we may assume
	that $h_k \weaklystar h$ in $X$
	for some $h \in X$.

	In case $h = 0$, we have
	\begin{align*}
		\hat c
		&=
		\lim_{k \to \infty} \frac{ \Phi_k(\bar x + t_k h_k)}{ \frac12 t_k^2 }
		\\&
		=
		\lim_{k \to \infty}
		\frac{
			G(\bar x + t_k h_k)
			-
			G(\bar x)
			-
			t_k \dual{h_k}{\bar\xi}_Y
			+
			\frac12 t_k^2 F''(\bar x)h_k^2
			+
			\frac12 t_k^2 \sum_{j = 1}^k \dual{h_k}{\xi_j}_Y^2
		}{ \frac12 t_k^2 }
		\\&
		\ge
		\liminf_{k \to \infty}
		\frac{
			G(\bar x + t_k h_k)
			-
			G(\bar x)
			-
			t_k \dual{h_k}{\bar\xi}_Y
			+
			\frac12 t_k^2 F''(\bar x)h_k^2
		}{ \frac12 t_k^2 }
		>
		0,
	\end{align*}
	where we used \eqref{eq:NDC} in the final estimate.

	In case $h \ne 0$,
	we have
	\begin{equation*}
		\liminf_{k \to \infty}
		\frac{
			G(\bar x + t_k h_k)
			-
			G(\bar x)
			-
			t_k \dual{h_k}{\bar\xi}_Y
		}{ \frac12 t_k^2 }
		\ge
		\sd^2 G(\bar x, \bar\xi; h)
		>
		-\infty
	\end{equation*}
	by assumption
	and
	\begin{equation*}
		\liminf_{k \to \infty} F''(\bar x) h_k^2
		\ge
		F''(\bar x) h^2
		>
		-\infty
	\end{equation*}
	by the sequential weak-$\star$ lower semicontinuity of $h \mapsto F''(\bar x)h^2$.
	Further, the density of $\seq{\xi_j} \subset Y$ implies
	\begin{equation*}
		\lim_{k \to \infty}
		\sum_{j = 1}^k \dual{h_k}{\xi_j}_Y^2
		=
		\infty
		.
	\end{equation*}
	Consequently,
	\begin{align*}
		\hat c
		&=
		\lim_{k \to \infty} \frac{ \Phi_k(\bar x + t_k h_k)}{ \frac12 t_k^2 }
		\\&
		=
		\lim_{k \to \infty}
		\frac{
			G(\bar x + t_k h_k)
			-
			G(\bar x)
			-
			t_k \dual{h_k}{\bar\xi}_Y
			+
			\frac12 t_k^2 F''(\bar x)h_k^2
			+
			\frac12 t_k^2 \sum_{j = 1}^k \dual{h_k}{\xi_j}_Y^2
		}{ \frac12 t_k^2 }
		\\&
		\ge
		\liminf_{k \to \infty}
		\frac{
			G(\bar x + t_k h_k)
			-
			G(\bar x)
			-
			t_k \dual{h_k}{\bar\xi}_Y
		}{ \frac12 t_k^2 }
		+
		\liminf_{k \to \infty} F''(\bar x) h_k^2
		+
		\lim_{k \to \infty}
		\sum_{j = 1}^k \dual{h_k}{\xi_j}_Y^2
		=
		\infty.
	\end{align*}
	Hence, we have shown $\hat c > 0$
	and, as explained above, this finishes the proof.
\end{proof}
Note that \eqref{eq:SSC} implies
\(
	\sd^2 G(\bar x,\bar\xi; h)
	\ge
	F''(\bar x) h^2
	>
	-\infty
\)
for all $h \ne 0$
and this implies the additional assumption
$\sd^2 G(\bar x,\bar\xi; h) > -\infty$ for all $h \ne 0$
required in \cref{thm:characterization_of_NDC}.

If we assume the stronger continuity property
of $h \mapsto F''(\bar x)h^2$ from
\cref{thm:no-gap},
which holds for many interesting applications,
we can strengthen the result of \cref{thm:characterization_of_NDC}.
\begin{corollary}
	\label{cor:NDC_weak_cts}
	In addition to the assumption of \cref{thm:characterization_of_NDC},
	we assume that
	$X \ni h \mapsto F''(\bar x) h^2$
	is sequentially weak-$\star$ continuous.
	Then, \eqref{eq:NDC} is satisfied
	if and only if
	there exist $r \in (0,\infty]$, $c > 0$, $n \in \N$ and $\xi_1, \ldots, \xi_n \in Y$ such that
	\begin{equation}
		\label{eq:growth_of_G_without_F}
		G(x) - G(\bar x) - \dual{x - \bar x}{\bar\xi}_Y
		\ge
		\frac{c}{2} \norm{x - \bar x}_X^2
		-
		\frac12 \sum_{j = 1}^n \dual{x - \bar x}{\xi_j}_Y^2
		\quad
		\forall x \in B^X_r(\bar x)
	\end{equation}
	holds.
\end{corollary}
\begin{proof}
	Under the additional assumption, it is clear that \eqref{eq:NDC} is equivalent to
	\begin{equation*}
		\label{eq:NDCp}
		\tag{NDC'}
		\begin{aligned}
			&\text{for all $\seq{t_k} \subset (0,\infty)$, $\seq{h_k} \subset X$
				with $t_k \to 0$, $h_k \weaklystar 0$
			and $\norm{h_k}_X = 1$, we have}\\
			&\qquad
			\liminf_{k \to \infty} \parens[\bigg]{
				\frac1{t_k^2}
				\braces[\big]{
					G(\bar x + t_k h_k)
					-
					G(\bar x)
					-
					t_k \dual{h_k}{\bar\xi}_Y
				}
			}
			> 0
			.
		\end{aligned}
	\end{equation*}
	Consequently, the result follows by applying \cref{thm:characterization_of_NDC}
	with $F$ and $F''(\bar x)$ replaced by $-\bar\xi$ and $0$, respectively.
\end{proof}
Note that (the proof of) \cite[Lemma~4.6]{WachsmuthWalter2024}
already provides the result that \eqref{eq:growth_of_G_without_F}
implies \eqref{eq:NDC}
and this does not use the additional property of $F''(\bar x)$ from \cref{cor:NDC_weak_cts}.
Hence,
\eqref{eq:growth_of_G_without_F} implies \eqref{eq:growth_of_G}
under the assumptions of \cref{thm:characterization_of_NDC}.
Moreover, under the additional assumption from \cref{cor:NDC_weak_cts},
\eqref{eq:growth_of_G} and \eqref{eq:growth_of_G_without_F}
become equivalent to each other.

Another interesting situation appears if we apply \cref{thm:characterization_of_NDC}
for $G \equiv 0$.
In this setting, we recover the results of \cite{Harder2018}
and \cite[Section~4.3]{Harder2021}
that (weak-$\star$) Legendre forms can only exist on Hilbertizable spaces.
\begin{corollary}
	\label{cor:Legendre}
	Let us assume that the bounded bilinear form $F''(\bar x) \colon X \times X \to \R$
	is a weak-$\star$ Legendre form in the sense that
	$X \ni h \mapsto F''(\bar x) h^2$ is
	sequentially weak-$\star$ lower semicontinuous
	and for all sequences $\seq{h_k} \subset X$ and $h \in X$ we have the implication
	\begin{equation*}
		h_k \weaklystar h
		\quad\text{and}\quad
		F''(\bar x) h_k^2 \to F''(\bar x) h^2
		\qquad\Rightarrow\qquad
		h_k \to h
		.
	\end{equation*}
	Then,
	there exist $c > 0$, $n \in \N$ and $\xi_1, \ldots, \xi_n \in Y$ such that
	\begin{equation}
		\label{eq:growth_of_Fpp}
		\frac12 F''(\bar x) h^2
		\ge
		\frac{c}{2} \norm{h}_X^2
		-
		\frac12 \sum_{j = 1}^n \dual{h}{\xi_j}_Y^2
		\quad
		\forall h \in X
	\end{equation}
	holds.
	In particular, $X$ is isomorphic to a Hilbert space.
\end{corollary}
\begin{proof}
	We check that \eqref{eq:NDC} is satisfied with $G \equiv 0$ and $\bar\xi = 0$.
	That is, given sequences $\seq{t_k} \subset (0,\infty)$ and $\seq{h_k} \subset X$
	with $t_k \to 0$, $h_k \weaklystar 0$ and $\norm{h_k}_X = 1$,
	we have to verify
	$\liminf_{k \to \infty} F''(\bar x) h_k^2 > 0$.
	The assumed semicontinuity already provides us with
	$\liminf_{k \to \infty} F''(\bar x) h_k^2 \ge 0$.
	If the limit inferior were $0$, we would have
	$h_{n_k} \weaklystar 0$ and $\lim_{k \to \infty} F''(\bar x) h_{n_k}^2 \to 0 = F''(\bar x) 0^2$
	along a corresponding subsequence,
	and, by assumption, this implies $h_{n_k} \to 0$ in $X$
	which contradicts $\norm{h_{n_k}}_X = 1$.
	Hence, \eqref{eq:NDC} holds.

	Now, the application of \cref{thm:characterization_of_NDC}
	implies \eqref{eq:growth_of_Fpp}.
	W.l.o.g., we may assume that $F''(\bar x)$
	is symmetric,
	since $h \mapsto F''(\bar x) h^2$ only depends on the symmetric part
	of $F''(\bar x)$.
	Consequently,
	\begin{equation*}
		\innerprod{h_1}{h_2}_X
		:=
		F''(\bar x)[h_1, h_2]
		+
		\sum_{j = 1}^n \dual{h_1}{\xi_j}_Y \dual{h_2}{\xi_j}_Y
	\end{equation*}
	defines an inner product on $X$ whose associated norm is equivalent to the norm of $X$.
	This yields the claim.
\end{proof}

We collect the findings of this section.
\begin{theorem}
	\label{thm:SSC}
	Under one of the additional assumptions of \cref{thm:no-gap},
	the following conditions are equivalent.
	\begin{enumerate}
		\item
			\eqref{eq:QG} holds with $c, r > 0$.
		\item
			\eqref{eq:SSC} and \eqref{eq:NDC} hold.
		\item
			\eqref{eq:SSC} and \eqref{eq:growth_of_G}
			hold
			with $r \in (0,\infty]$, $c > 0$, $n \in \N$, and $\xi_1, \ldots, \xi_n \in Y$.
	\end{enumerate}
	If $h \mapsto F''(\bar x) h^2$ is sequentially weak-$\star$ continuous,
	then also
	\begin{enumerate}[resume]
		\item
			\eqref{eq:SSC} and \eqref{eq:growth_of_G_without_F}
			hold
			with $r \in (0,\infty]$, $c > 0$, $n \in \N$, and $\xi_1, \ldots, \xi_n \in Y$.
	\end{enumerate}
	is equivalent to the previous statements.
\end{theorem}
In a certain sense,
the condition that $h \mapsto F''(\bar x) h^2$ is sequentially weak-$\star$ continuous
means that
$F''(\bar x)$ (or $F$ itself)
does not has a regularizing effect in the problem \eqref{eq:problem}.
In particular, it rules out Legendre forms on Hilbert spaces.
In this case, all the regularization has to come from $G$
and, therefore, \eqref{eq:growth_of_G_without_F}
has to be satisfied whenever we want to achieve
quadratic growth in the space $X$.

In the final result of this section,
we want to demonstrate another crucial consequence of \eqref{eq:growth_of_G_without_F}.
It shows that sequences having a bounded second-order difference quotient
can be bounded in norm by an explicit expression.
\begin{proposition}
	\label{prop:consequence_of_growth}
	Suppose that \eqref{eq:growth_of_G_without_F} is satisfied.
	Then,
	for all sequences $\seq{t_k} \subset (0,\infty)$ and $\seq{h_k} \subset X$
	satisfying $t_k \to 0$, $h_k \weaklystar h$ in $X$ for some $h \in X$
	we have
	\begin{equation*}
		\limsup_{k \to \infty} \norm{ h_k }_X^2
		\le
		\frac{1}{c}
		\parens[\bigg]{
		L
		+
		\sum_{j = 1}^n \dual{h}{\xi_j}_Y^2
	},
	\end{equation*}
	where we used $c$, $n$ and $\xi_j$, $j = 1,\ldots,n$,
	from \eqref{eq:growth_of_G_without_F},
	and
	\begin{equation*}
		L
		:=
		\limsup_{k \to \infty}
		\frac{ G(\bar x + t_k h_k) - G(\bar x) - t_k \dual{h_k}{\bar\xi}_Y }{\frac12 t_k^2}
		.
	\end{equation*}
\end{proposition}
\begin{proof}
	From \eqref{eq:growth_of_G_without_F} and
	$t_k h_k \to 0$ in $X$
	we get
	\begin{equation*}
		G(\bar x + t_k h_k) - G(\bar x) - t_k \dual{h_k}{\bar\xi}_Y
		\ge
		\frac{c}{2} \norm{t_k h_k}_X^2
		-
		\frac12 \sum_{j = 1}^n \dual{t_k h_k}{\xi_j}_Y^2
	\end{equation*}
	for all $k$ large enough.
	Consequently,
	\begin{align*}
		L
		&=
		\limsup_{k \to \infty}
		\frac{ G(\bar x + t_k h_k) - G(\bar x) - t_k \dual{h_k}{\bar\xi}_Y }{\frac12 t_k^2}
		\\&
		\ge
		\limsup_{k \to \infty} \braces*{
			c \norm{h_k}_X^2
			-
			\sum_{j = 1}^n \dual{h_k}{\xi_j}_Y^2 }
		=
		\limsup_{k \to \infty} \braces*{
			c \norm{h_k}_X^2
		}
		-
		\sum_{j = 1}^n \dual{h}{\xi_j}_Y^2
		.
	\end{align*}
	This yields the claim.
\end{proof}

\section{Applications to weak-\texorpdfstring{$\star$}{*} second subderivatives}
\label{sec:subderivatives}

In this section, we want to apply the findings of
\cref{sec:epi,sec:subdifferentials}
to simplify the computation of the weak-$\star$ second subderivative.
To this end, we will utilize the condition \eqref{eq:growth_of_G_without_F}
which is (under some assumptions)
equivalent to \eqref{eq:NDC}, see \cref{cor:NDC_weak_cts}.

In order to apply \cref{thm:main_result_conjugate},
we need to assume that the functional $G$ is actually a conjugate function.
In particular, $G$ has to be a convex function.
\begin{assumption}[Standing assumption]
	\label{asm:standing_application}
	In addition to \cref{asm:standing_subderivatives},
	we assume the existence of $J \in \Gamma_0(Y)$
	such that $G = J\conjugate$.
	Moreover, we assume that $\bar\xi \in \partial G(\bar x)$.
\end{assumption}
Recall that we introduced the definition
$\bar\xi := -F'(\bar x)$ in \cref{asm:standing_subderivatives}.
Thus, the last assumption
is precisely the first-order optimality condition for
problem \eqref{eq:problem} at $\bar x$,
see, e.g., \cite[Theorem~2.12]{BorchardWachsmuth2023}.
By the Fenchel--Young identity,
this condition is also equivalent to
$\bar x \in \partial J(\bar\xi)$.
Further,
it implies
$\sd^2 G(\bar x, \bar\xi; \cdot) \ge 0$.

To explain the connection between \cref{sec:epi}
and the definition of the weak-$\star$ subderivative,
we fix a sequence $\seq{t_k} \subset (0,\infty)$
and define
the second-order difference quotients
\begin{equation*}
	\Delta^2_{t_k} G(\bar x, \bar \xi; h)
	:=
	\frac{ G(\bar x + t_k h) - G(\bar x) - t_k \dual{h}{\bar\xi}_Y }{\frac12 t_k^2}
\end{equation*}
and
\begin{equation*}
	\Delta^2_{t_k} J(\bar\xi, \bar x; y)
	:=
	\frac{ J(\bar \xi + t_k y) - J(\bar \xi) - t_k \dual{\bar x}{y}_Y }{\frac12 t_k^2}
	.
\end{equation*}
The assumption that $G$ is the convex conjugate of $J$
transfers to these difference quotients.
Indeed,
a straightforward computation shows that the (scaled) dual difference quotient
$\frac12 \Delta^2_{t_k} G(\bar x, \bar \xi; \cdot) \colon X \to (-\infty, \infty]$
is the convex conjugate of the (scaled) primal difference quotient
$\frac12 \Delta^2_{t_k} J(\bar\xi, \bar x; \cdot) \colon Y \to (-\infty, \infty]$.

Now, the crucial point is that the weak-$\star$ second subderivative
$\sd^2 G(\bar x, \bar\xi; \cdot)$
is the largest function
such that
\eqref{BPfcw}
holds with
$f\conjugate = \sd^2 G(\bar x, \bar\xi; \cdot)$
and
$f_k\conjugate = \Delta^2_{t_k} G(\bar x, \bar \xi; \cdot)$
for all sequences $\seq{t_k} \subset (0,\infty)$ with $t_k \to 0$.
Or,
put differently,
the weak-$\star$ epidifferentiability of $G$ at $\bar x$ for $\bar\xi$
is related to the $\Gamma$-convergence
of the difference quotients
$\Delta^2_{t_k} G(\bar x, \bar \xi; \cdot)$
w.r.t.\ the weak-$\star$ topology.

In the light of the results of \cref{sec:epi}, it makes sense to
investigate a strong second subderivative of the preconjugate functional.
\begin{definition}
	\label{def:strong-subderivative}
	The strong second subderivative
	$\ds^2 J(\bar\xi, \bar x; y) \colon Y \to [0,\infty]$
	of $J$ at $\bar\xi$ for $\bar x$ is defined via
	\begin{equation*}
		\ds^2 J(\bar\xi, \bar x; y)
		=
		\inf\set*{
			\liminf_{k \to \infty}
			\frac{ J(\bar\xi + t_k y_k) - J(\bar\xi) - t_k \dual{\bar x}{y_k}_Y }{\frac12 t_k^2}
			\given
			t_k \searrow 0, \; y_k \to y
		}
	\end{equation*}
	for all $y \in Y$.
	Here, the infimum is taken w.r.t.\ all sequences $\seq{t_k} \subset (0,\infty)$,
	$\seq{y_k} \subset Y$ such that
	$t_k \to 0$ and $y_k \to y$.

	Further, we say that $J$ is
	strongly-strongly twice epidifferentiable
	at $\bar\xi$ for $\bar x$ if for every $y \in Y$ and
	every sequence $\seq{t_k} \subset (0,\infty)$ with $t_k \to 0$
	there exists a sequence $\seq{y_k} \subset Y$ such that
	\begin{equation*}
		\ds^2 J(\bar\xi, \bar x; y)
		=
		\lim_{k \to \infty}
		\frac{ J(\bar\xi + t_k y_k) - J(\bar\xi) - t_k \dual{\bar x}{y}_Y }{\frac12 t_k^2}
	\end{equation*}
	and $y_k \to y$.
\end{definition}
Note that the nonnegativity of $\ds^2 J(\bar\xi, \bar x; \cdot)$ follows
from the convexity of $J$ and $\bar x \in \partial J(\bar\xi)$.

Now, we are in position to apply \cref{thm:main_result_conjugate}
to the situation at hand.
\begin{theorem}
	\label{thm:main_result}
	In addition to \cref{asm:standing_application},
	we assume that $Q \in \Gamma_0(Y)$ is
	positively homogeneous of order $2$.
	Consider the following statements.
	\begin{enumerate}
		\item
			\label{thm:main_result:1}
			The functional $G$ is weak-$\star$ twice epidifferentiable at $\bar x$ for $\bar\xi$
			and
			$\frac12 \sd^2 G(\bar x, \bar\xi; \cdot) = Q\conjugate$.
		\item
			\label{thm:main_result:2b}
			For every pair $(y, h) \in \graph\partial Q$
			and every sequence $\seq{t_k} \subset (0,\infty)$ with $t_k \to 0$,
			there exist sequences $\seq{y_k} \subset Y$, $\seq{h_k} \subset X$
			with
			$y_k \to y$ in $Y$, $h_k \weaklystar h$ in $X$
			and
			$\bar\xi + t_k y_k \in \partial G(\bar x + t_k h_k)$.

		\item
			\label{thm:main_result:2}
			In addition to \ref{thm:main_result:2b},
			we have the following property:
			Given
			sequences
			$\seq{t_k} \subset (0,\infty)$,
			$\seq{y_k} \subset Y$
			and
			$\seq{h_k} \subset X$
			with
			$t_k \to 0$, $y_k \to y$ in $Y$ and $h_k \weaklystar h$ in $X$
			and
			$\bar\xi + t_k y_k \in \partial G(\bar x + t_k h_k)$ for all $k$
			implies
			$(y, h) \in \graph\partial Q$.

		\item
			\label{thm:main_result:3}
			The functional $J$ is strongly-strongly twice epidifferentiable at $\bar\xi$ for $\bar x$
			and
			$\frac12 \ds^2 J(\bar\xi, \bar x; \cdot) = Q$.
	\end{enumerate}
	Then,
	\ref{thm:main_result:1}
	$\Rightarrow$
	\ref{thm:main_result:2b}
	$\Leftrightarrow$
	\ref{thm:main_result:2}
	$\Rightarrow$
	\ref{thm:main_result:3}.
	Moreover, if
	\eqref{eq:growth_of_G_without_F}
	holds,
	then
	\ref{thm:main_result:3}
	$\Rightarrow$
	\ref{thm:main_result:1}.
\end{theorem}
\begin{proof}
	We define $f := Q$.
	We will check
	that each of the conditions of the present theorem is satisfied
	if and only if
	the corresponding condition of
	\cref{thm:main_result_conjugate}
	is satisfied with
	$f_k := \frac12 \Delta_{t_k}^2 J(\bar\xi, \bar x; \cdot)$
	for all sequences $\seq{t_k} \subset (0,\infty)$ with $t_k \to 0$.

	Indeed,
	\itemref{thm:main_result:1} can be rephrased as follows.
	For every $\seq{t_k} \subset (0,\infty)$ with $t_k \to 0$,
	the next two conditions are satisfied
	(with the above choice of $f_k$).
	\begin{itemize}
		\item
			For every $\seq{h_k} \subset X$ and $h \in X$
			with $h_k \weaklystar h$ in $X$,
			we have
			\begin{align*}
				f\conjugate(h)
				=
				Q\conjugate(h)
				\le
				\liminf_{k \to \infty}
				\braces*{
					\frac12 \Delta_{t_k}^2 G(\bar x, \bar\xi; h_k)
				}
				&=
				\liminf_{k \to \infty}
				\parens*{
					\frac12 \Delta_{t_k}^2 J(\bar\xi, \bar x; \cdot)
				}\conjugate(h_k)
				\\&
				=
				\liminf_{k \to \infty} f_k\conjugate(h_k)
				.
			\end{align*}
		\item
			For every $h \in X$, there exists $\seq{h_k} \subset X$
			with $h_k \weaklystar h$ in $X$
			such that
			\begin{align*}
				f\conjugate(h)
				=
				Q\conjugate(h)
				=
				\lim_{k \to \infty}
				\braces*{
					\frac12 \Delta_{t_k}^2 G(\bar x, \bar\xi; h_k)
				}
				&=
				\lim_{k \to \infty}
				\parens*{
					\frac12 \Delta_{t_k}^2 J(\bar\xi, \bar x; \cdot)
				}\conjugate(h_k)
				\\&
				=
				\lim_{k \to \infty} f_k\conjugate(h_k)
				.
			\end{align*}
	\end{itemize}
	Note that these conditions are precisely
	\eqref{BPfcw} and \eqref{RSfcw},
	i.e.,
	\itemref{thm:main_result_conjugate:1}.
	This shows
	\itemref{thm:main_result:1}
	is equivalent to
	\itemref{thm:main_result_conjugate:1}
	being satisfied for all null sequences $\seq{t_k}$.

	The same reasoning shows the equivalence of
	\itemref{thm:main_result:3}
	and
	\itemref{thm:main_result_conjugate:3}
	(for all null sequences $\seq{t_k}$).

	Next, we briefly argue that the normalization condition \eqref{NC}
	is trivially satisfied
	with $(y_0,x_0) = (y_{0,k}, x_{0,k}) = (0,0)$,
	since
	the positive homogeneity of $Q$ implies
	$Q(0) = 0$ and $0 \in \partial Q(0)$,
	while
	$\Delta_{t_k}^2 J(\bar\xi, \bar x; 0) = 0$
	and
	$0 \in \partial (\Delta_{t_k}^2 J(\bar\xi, \bar x; \cdot))(0)$
	follow from the definition of the difference quotient
	and $\bar x \in \partial J(\bar\xi)$.

	Next, we investigate the condition
	$(y_k, h_k) \in \graph\partial f_k$
	appearing in
	\eqref{BPpf} and \eqref{RSpf}
	mentioned in
	\itemref{thm:main_result_conjugate:2b}
	and
	\ref{thm:main_result_conjugate:2}.
	This condition
	is equivalent to
	\begin{equation*}
		h_k
		\in
		\partial \parens*{\frac12 \Delta_{t_k}^2 J(\bar\xi, \bar x; \cdot)}(y_k)
		=
		\frac{
			\partial (J(\bar\xi + t_k \cdot))(y_k) - t_k \bar x
		}{t_k^2}
		=
		\frac{
			t_k \partial J(\bar\xi + t_k y_k) - t_k \bar x
		}{t_k^2}
	\end{equation*}
	and, thus, it is equivalent to
	\begin{equation*}
		\bar x + t_k h_k
		\in
		\partial J(\bar\xi + t_k y_k)
		.
	\end{equation*}
	With this, we see that
	\itemref{thm:main_result:2b}
	is equivalent to
	\itemref{thm:main_result_conjugate:2b}
	(for all null sequences $\seq{t_k}$)
	and that
	\itemref{thm:main_result:2}
	is equivalent to
	\itemref{thm:main_result_conjugate:2}
	(for all null sequences $\seq{t_k}$).

	Thus, we can invoke
	\cref{thm:main_result_conjugate}
	to get the first chain of implications from the assertion.

	Finally, the additional assumption \eqref{eq:growth_of_G_without_F}
	implies
	\begin{align*}
		f_k\conjugate(h)
		&=
		\frac12 \Delta_{t_k}^2 G(\bar x, \bar\xi; h)
		=
		\frac{ G(\bar x + t_k h) - G(\bar x) - \dual{t_k h}{\bar\xi}_Y}{t_k^2}
		\\&
		\ge
		\frac{
			\frac{c}{2} \norm{t_k h}_X^2
			-
			\frac12 \sum_{j = 1}^n \dual{t_k h}{\xi_j}_Y^2
		}{t_k^2}
		=
			\frac{c}{2} \norm{h}_X^2
			-
			\frac12 \sum_{j = 1}^n \dual{h}{\xi_j}_Y^2
		\end{align*}
		whenever $\norm{t_k h}_X \le r$.
		That is,
		the inequality in \itemref{prop:weakstar_dual_recovery:2}
		is satisfied with $r_k = r/t_k$.
		Moreover, the convexity of $J$ and $\bar x \in \partial J(\bar\xi)$
		imply $f_k(0) = f_k\conjugate(0) = 0$ and $f_k \ge 0$.
		Hence, the inequality \eqref{eq:growth_of_G_without_F} implies that the
		additional assumption from \cref{prop:weakstar_dual_recovery} is satisfied.
		Consequently, the last assertion in \cref{thm:main_result_conjugate}
		implies
		``\ref{thm:main_result:3} $\Rightarrow$ \ref{thm:main_result:1}''.
\end{proof}

Roughly speaking,
\itemref{thm:main_result:2}
asserts
that $\partial Q$
is the graphical derivative
of $\partial J$ at $\bar\xi$ for $\bar x$
or,
equivalently,
that $\partial Q\conjugate$
is the graphical derivative
of $\partial G$ at $\bar x$ for $\bar\xi$
(if we restrict the latter subdifferentials to $Y$).

We mention that \cref{thm:main_result}
is very similar to the results
\cite[Theorems~2.5(b), 3.9]{Do1992}
in reflexive spaces.
However,
a crucial difference is that
the results of Do
use (what we call)
strong twice epidifferentiability
for $J$ and $G$.
This comes from Mosco's theorem on Mosco convergence
in reflexive spaces.
Since we replace this theorem by \cref{thm:main_result_conjugate}
in our proof,
we have to use
weak-$\star$ twice epidifferentiability for $G$
and
strong-strong twice epidifferentiability for $J$.
Similarly,
the mode of convergence of the difference quotients of the subdifferential in
\itemref{thm:main_result:2b}
and
\ref{thm:main_result:2}
is different from what is used in \cite[Theorem~3.9]{Do1992}.
We emphasize that Do's approach only applies in reflexive spaces,
since it relies on the theorems by Mosco and Attouch,
which are only available in reflexive spaces,
see, in particular, \cite{BeerBorwein1990,AttouchBeer1993}.

\Cref{thm:main_result}
has a great potential to simplify the computation of
weak-$\star$ second subderivatives considerably.
Indeed,
in the references
\cite{WachsmuthWachsmuth2021,WachsmuthWalter2024,BorchardMeyerWachsmuth2026:1},
the authors first proved that the preconjugate $J$
is strongly-strongly twice epidifferentiable.
Afterwards, additional and lengthy calculations
are carried out in order to prove that $G$ is weak-$\star$ twice epidifferentiable.

Let us have a closer look at \cite{WachsmuthWachsmuth2021}.
Therein, the strong-strong twice epidifferentiability of $J$
is shown in Section~3.1.
Afterwards, a certain approximation property of subdifferentials
of the strong second subderivative
(similar to \itemref{thm:main_result:2b})
is verified in Section~3.2.
Eventually, Section~4.2
proves the weak-$\star$ twice epidifferentiability of $G$.
Note that
the contents of these Sections~3.2 and 4.2
can be (almost) replaced by the application of \cref{thm:main_result}.
We emphasize that the arguments in \cite{WachsmuthWachsmuth2021}
also provide that $G$ is strictly twice epidifferentiable,
i.e., the recovery sequences also satisfy $\norm{h_k}_X \to \norm{h}_X$.
However, up to now we do not know any
theoretical results which make use of the strict twice epidifferentiability
of $G$.

The calculations in \cite{BorchardMeyerWachsmuth2026:1}
follow roughly the same pattern as in \cite{WachsmuthWachsmuth2021}
and, thus, can be shortened significantly.
However, the structure of the functional $G$ in \cite{WachsmuthWalter2024}
allows for a direct construction of recovery sequences.

\begin{remark}
	\label{rem:on_subgradients}
	Finally,
	we would like to mention that the condition \itemref{thm:main_result:2}
	also has an important application.
	Indeed, we mention the result
	\cite[Theorem~6.11]{WachsmuthWalter2025:1}
	which can be used to show that \eqref{eq:SSC}
	not only implies \eqref{eq:QG}
	but also the so-called strong metric subregularity
	of the set-valued map $x \mapsto F'(x) + \partial G(x)$
	as well as the
	important
	Polyak--Łojasiewicz--Kurdyka condition.
	In particular, its assumption (iii)
	requires
	that $t_k \to 0$, $h_k \weaklystar h$, $y_k \to y$
	and $\bar\xi + t_k y_k \in \partial G(\bar x + t_k h_k)$
	implies
	$\dual{h}{y}_Y \ge \sd^2 G(\bar x, \bar\xi; h)$.
	However, from \itemref{thm:main_result:2}
	we already get $y \in \partial(\frac12 \sd^2 G(\bar x, \bar\xi; \cdot))(h)$
	and,
	since the second subderivative is positively homogeneous of order $2$,
	this implies
	$\dual{h}{y}_Y = \sd^2 G(\bar x, \bar\xi; h)$.
	Consequently, the condition
	\cite[Theorem~6.11(iii)]{WachsmuthWalter2025:1}
	already follows from the weak-$\star$ twice epidifferentiability of $G$
	and
	some of the calculations in \cite[Section~7.2]{WachsmuthWalter2025:1}
	can be simply replaced by an application of
	\cref{thm:main_result}.
\end{remark}

\section{Weak-\texorpdfstring{$\star$}{*} second subderivatives in compatible spaces}
\label{sec:all_are_equal}

The aim of this section is to formalize and generalize
the results from \cite[Section~4.2]{Wachsmuth2026:1}.
That is,
we prove that the weak-$\star$ subderivative of a functional
does not change if (roughly speaking)
we switch between reasonable spaces.

Very often in applications,
the space $X$ appearing in the problem \eqref{eq:problem}
is not part of the application itself
but arises in the modeling process.
Typically, the functional $G$ from \eqref{eq:problem}
is initially defined to be real-valued on a certain set $\Gamma$,
i.e., $G \colon \Gamma \to \R$.
This set $\Gamma$ has then to be embedded into some dual space $X$
in order to apply the second-order theory from \cref{sec:SSC}.
For example, in optimal control problems, we are given the admissible set
\begin{equation*}
	\Uad := \set{ u \colon \Omega \to \R \given \text{$u$ is measurable and $u_a \le u \le u_b$ a.e.}}
\end{equation*}
for some open set $\Omega \subset \R^d$ and (for simplicity) constants $u_a < u_b$.
If we do not add any further regularization,
the functional $G$ is just the zero functional on the set $\Uad$.
Of course, $\Uad$ is embedded into $L^\infty(\Omega)$ and
we could replace $G \colon \Uad \to \R$ by
$G \colon L^\infty(\Omega) \to (-\infty,\infty]$
via
\begin{equation*}
	G(u)
	:=
	\begin{cases}
		0 & \text{if } u \in \Uad \\
		\infty & \text{else}.
	\end{cases}
\end{equation*}
However, it has been shown in
\cite{ChristofWachsmuth2017:1}
that $X = L^\infty(\Omega)$ is not a valid choice
(in particular, \eqref{eq:NDC} is violated)
and should be replaced by $X = \MM(\Omega) = C_0(\Omega)\dualspace$.
Quite recently,
it has been shown that \eqref{eq:NDC} is also satisfied
in certain dual Bessel potential spaces,
see \cite{Wachsmuth2026:1}.
Surprisingly, the weak-$\star$ second subderivative
and the validity of the weak-$\star$ second epidifferentiability
do not change when we replace $X = \MM(\Omega)$
by a dual Bessel space $X$.
In particular, we do not have to repeat (or adapt)
the lengthy computations leading to the precise form
of the weak-$\star$ second subderivative.

In order to generalize these findings,
we pose the following assumption.
\begin{assumption}
	\label{asm:standing_equivalent}
	We are given two separable Banach spaces $Y_1$ and $Y_2$
	such that the vector space operations of $Y_1$ and $Y_2$
	are identical on $Y_1 \cap Y_2$
	and such that $Y_1 \cap Y_2$ is dense both in $Y_1$ and in $Y_2$.

	We denote by $X_1 := Y_1\dualspace$ and $X_2 := Y_2\dualspace$
	the corresponding dual spaces.
	Further, we set $Y := Y_1 \cap Y_2$ and
	\begin{equation*}
		X
		:=
		\set{
			\ell \colon Y_1 \cap Y_2 \to \R
			\given
			\text{$\ell$ is continuous w.r.t.\ $\norm{\cdot}_{Y_1}$ and w.r.t.\ $\norm{\cdot}_{Y_2}$}
		}
		.
	\end{equation*}
\end{assumption}
At this point, we would like to emphasize that we do not equip $Y$ with a norm
and that $X$ is not the dual space of $Y$.
However, we nevertheless define
\begin{equation*}
	\dual{x}{y}_Y
	:=
	x(y)
	\qquad
	\forall x \in X, y \in Y
	.
\end{equation*}

With some abuse of notation,
we say that a functional $h \in X_1$ belongs to $X$
if $h$ restricted to $Y$ belongs to $X$.
Note that this is not harmful since the values of $h \colon Y_1 \to \R$ on $Y_1$
can be reconstructed from the values of its restriction on $Y = Y_1 \cap Y_2$
by density (of $Y$ in $Y_1$) and continuity (of $h$ w.r.t.\ $\norm{\cdot}_{Y_1}$).
In a similar spirit, we define $h \in X$ if $h \in X_2$ is given.
Consequently, we have
\begin{equation*}
	X = X_1 \cap X_2
\end{equation*}
under this identification.
Moreover,
\begin{equation*}
	\dual{x}{y}_{Y_i}
	=
	\dual{x}{y}_Y
\end{equation*}
whenever $i \in \set{1,2}$, $x \in X \subset X_i$ and $y \in Y \subset Y_i$ are given.

This being said,
we specify the assumptions on the functional under consideration.
\begin{assumption}
	\label{asm:standing_equivalent_ctd}
	In addition to \cref{asm:standing_equivalent},
	we are given a functional $G \colon X \to (-\infty,\infty]$.
	This functional is extended to functionals
	$G_1 \colon X_1 \to (-\infty,\infty]$ and
	$G_2 \colon X_2 \to (-\infty,\infty]$
	by extending them with $\infty$ on $X_1 \setminus X$ and $X_2 \setminus X$, respectively.
	Further, we fix $\bar x \in X$ and $\bar\xi \in Y$.
\end{assumption}

In the first lemma,
we show that if $G_i$ satisfies the condition \eqref{eq:growth_of_G_without_F}
on the space $X_i$,
the points $\xi_1, \ldots, \xi_n$ can be chosen from $Y$.
\begin{lemma}
	\label{lem:growth_of_G_Y}
	Suppose that \cref{asm:standing_equivalent_ctd} holds.
	Let $i \in \set{1,2}$ be given such that $G_i$ satisfies \eqref{eq:growth_of_G_without_F},
	i.e., there exist
	$r \in (0,\infty]$, $c > 0$ and $\xi_1, \ldots, \xi_n \in Y_i$ such that
	\begin{equation*}
		G_i(x) - G_i(\bar x) - \dual{x - \bar x}{\bar\xi}_{Y_i}
		\ge
		\frac{c}{2} \norm{x - \bar x}_{X_i}^2
		-
		\frac12 \sum_{j = 1}^n \dual{x - \bar x}{\xi_j}_{Y_i}^2
		\quad
		\forall x \in B^{X_i}_r(\bar x)
		.
	\end{equation*}
	Then, for every $\hat c \in (0,c)$ there exist
	$\hat\xi_1, \ldots, \hat\xi_n \in Y$ such that
	\begin{equation*}
		G_i(x) - G_i(\bar x) - \dual{x - \bar x}{\bar\xi}_{Y_i}
		\ge
		\frac{\hat c}{2} \norm{x - \bar x}_{X_i}^2
		-
		\frac12 \sum_{j = 1}^n \dual{x - \bar x}{\hat \xi_j}_{Y_i}^2
		\quad
		\forall x \in B^{X_i}_r(\bar x)
		.
	\end{equation*}
	In the case that $r = \infty$ or if $G$ is convex,
	we even have
	\begin{equation}
		\label{eq:strengthening_growth_of_G_without_F}
		G_i(x) - G_i(\bar x) - \dual{x - \bar x}{\bar\xi}_{Y_i}
		\ge
		\tilde c \min\set*{ \frac12 \norm{x - \bar x}_{X_i}^2, 1 }
		-
		\frac12 \sum_{j = 1}^n \dual{x - \bar x}{\hat \xi_j}_{Y_i}^2
		\quad
		\forall x \in X_i
	\end{equation}
	for some $\tilde c > 0$.
\end{lemma}
\begin{proof}
	Due to the density of $Y$ in $Y_i$,
	we can choose $\hat\xi_j$ such that
	\begin{equation*}
		\norm{\hat\xi_j - \xi_j}_{Y_i}
		\le
		\delta
	\end{equation*}
	for any $\delta > 0$.
	Consequently,
	\begin{align*}
		\abs*{
			\dual{x - \bar x}{\xi_j}_{Y_i}^2
			-
			\dual{x - \bar x}{\hat\xi_j}_{Y_i}^2
		}
		&=
		\abs*{
			\dual{x - \bar x}{\xi_j + \hat\xi_j}_{Y_i}
		}
		\abs*{
			\dual{x - \bar x}{\xi_j - \hat \xi_j}_{Y_i}
		}
		\\
		&\le
		\norm{\xi_j + \hat\xi_j}_{Y_i}
		\norm{\xi_j - \hat\xi_j}_{Y_i}
		\norm{x - \bar x}_{X_i}^2
		\\
		&\le
		(\norm{\xi_j}_{Y_i} + \delta)
		\delta 
		\norm{x - \bar x}_{X_i}^2
		.
	\end{align*}
	Now the first claim follows by choosing $\delta > 0$ small enough.

	In case that $r = \infty$,
	we clearly get the final inequality with $\tilde c = \hat c$.
	On the other hand, if $G$ is convex and $r < \infty$, we have
	\begin{equation*}
		G_i(x) - G_i(\bar x) - \dual{x - \bar x}{\bar\xi}_{Y_i}
		+
		\frac12 \sum_{j = 1}^n \dual{x - \bar x}{\hat \xi_j}_{Y_i}^2
		\ge
		\frac{\hat c}{2} r^2
	\end{equation*}
	for all $x \in X_i$ with $\norm{x - \bar x}_{X_i} = r$.
	Moreover, the left-hand side of this inequality
	is a convex function w.r.t.\ $x \in X_i$
	and the value at $x = \bar x$ equals $0$.
	By convexity, this implies
	\begin{equation*}
		G_i(x) - G_i(\bar x) - \dual{x - \bar x}{\bar\xi}_{Y_i}
		+
		\frac12 \sum_{j = 1}^n \dual{x - \bar x}{\hat \xi_j}_{Y_i}^2
		\ge
		\frac{\hat c}{2} r^2
	\end{equation*}
	for all $x \in X_i$ with $\norm{x - \bar x}_{X_i} > r$.
	The final claimed inequality with
	$\tilde c = \hat c \min\set[\big]{ 1, \frac{1}{2} r^2}$ follows.
\end{proof}
If \eqref{eq:growth_of_G_without_F} is satisfied for both,
$G_1$ and $G_2$,
we can use the same sequence $\seq{\hat\xi_j} \subset Y$.
Indeed, from \cref{lem:growth_of_G_Y} we know that
for each of $G_1$ and $G_2$, we find an individual sequence from the space $Y$
and afterwards, we can simply concatenate both sequences.

As stated above,
we want to show that
$G_1$ and $G_2$ share the same weak-$\star$ second subderivative
and that
$G_1$ is weak-$\star$ twice epidifferentiable
if and only if $G_2$ has the same property.
To this end, the following lemma is the key result.
It asserts that a sequence which converges in one space
and which has a bounded second-order difference quotient
actually has to converge in the other space as well
if the strengthening \eqref{eq:strengthening_growth_of_G_without_F} of \eqref{eq:growth_of_G_without_F} holds.
\begin{lemma}
	\label{lem:limits_coincide}
	Let \cref{asm:standing_equivalent_ctd} be satisfied.
	Suppose that sequences $\seq{t_k} \subset (0,\infty)$ and $\seq{h_k} \subset X_1$
	are given such that $t_k \to 0$, $h_k \weaklystar h$ in $X_1$ for some $h \in X_1$
	and
	\begin{equation*}
		\limsup_{k \to \infty} 
		\frac{ G_1(\bar x + t_k h_k) - G_1(\bar x) - t_k \dual{h_k}{\bar\xi}_{Y_1} }{\frac12 t_k^2}
		<
		\infty
		.
	\end{equation*}
	Further, we assume that \eqref{eq:strengthening_growth_of_G_without_F}
	holds for $G_2$.
	Then, $h_k \in X$ for all $k$ large enough, $h \in X$
	and $h_k \weaklystar h$ in $X_2$.
\end{lemma}
\begin{proof}
	The assumption implies that $G_1(\bar x + t_k h_k) < \infty$
	for $k$ large enough.
	Consequently,
	since $G_1 = \infty$ on $X_1 \setminus X$,
	this implies
	$h_k \in X$ for $k$ large enough.
	Clearly, it is sufficient to consider the case that $h_k \in X$
	for all $k \in \N$.

	Dividing \eqref{eq:strengthening_growth_of_G_without_F} by $\frac12 t_k^2$
	yields
	\begin{equation*}
		\frac{G_2(\bar x + t_k h_k) - G_2(\bar x) - t_k \dual{h_k}{\bar\xi}_Y}{\frac12 t_k^2}
		\ge
		\tilde c \min\set*{ \norm{h_k}_{X_2}^2 , \frac{1}{\frac12 t_k^2}}
		-
		\sum_{j = 1}^n \dual{h_k}{\hat\xi_j}_Y^2
		.
	\end{equation*}
	Due to $h_k \weaklystar h$ in $X_1$ and $\hat \xi_j \in Y \subset Y_1$,
	the sum on the right-hand side converges towards a real number.
	Moreover, since we have
	$G_2(\bar x + t_k h_k) = G_1(\bar x + t_k h_k)$
	and
	$G_2(\bar x) = G_1(\bar x)$,
	the assumption implies that the left-hand side
	is bounded as $k \to \infty$.
	Thus,
	$\min\set[\big]{ \norm{h_k}_{X_2}^2 , \parens{\frac12 t_k^2}^{-1}}$
	has to be bounded as well
	and this implies
	the boundedness of $\seq{h_k}$ in $X_2$.

	Consequently, every subsequence of $\seq{h_k}$
	has a further subsequence denoted by $\seq{h_{k_n}}$
	such that $h_{k_n} \weaklystar \tilde h$ in $X_2$.
	For any $y \in Y$ we have
	\begin{equation*}
		\dual{h}{y}_{Y_1}
		=
		\lim_{n \to \infty} \dual{h_{k_n}}{y}_{Y_1}
		=
		\lim_{n \to \infty} \dual{h_{k_n}}{y}_{Y_2}
		=
		\dual{\tilde h}{y}_{Y_2}
		.
	\end{equation*}
	Consequently,
	the functionals $h$ and $\tilde h$ coincide on $Y = Y_1 \cap Y_2$
	and they are continuous on $Y$ w.r.t.\ the norm of $Y_1$ and the norm of $Y_2$.
	In particular,
	this shows $h \in X$
	and $h_{k_n} \weaklystar h$ in $X_2$.
	Since the limit $h$ is independent of the chosen subsequence,
	the usual subsequence-subsequence argument
	shows $h_k \weaklystar h$ in $X_2$
	for the entire sequence.
\end{proof}
Note that the validity of \eqref{eq:strengthening_growth_of_G_without_F}
does not change if
$G_i$ is replaced by $G$
and $x \in X_i$ is replaced by $x \in X$,
due to $G_i |_{X} = G$
and $G_i = \infty$ on $X_i \setminus X$.

By repeating the argument of
\cite[Lemma~4.6]{Wachsmuth2026:1},
we arrive at the final result of this section.
\begin{theorem}
	\label{thm:equivalent_subderivatives}
	Let \cref{asm:standing_equivalent_ctd} be satisfied.
	Further, we assume that
	there exist $\tilde c > 0$, $n \in \N$
	and $\hat\xi_j \in Y$, $j = 1,\ldots,n$ such that
	\begin{align*}
		G(x) - G(\bar x) - \dual{x - \bar x}{\bar\xi}_{Y}
		&\ge
		\tilde c \min\set*{ \frac12 \norm{x - \bar x}_{X_1}^2, 1 }
		-
		\frac12 \sum_{j = 1}^n \dual{x - \bar x}{\hat \xi_j}_{Y}^2
		\quad
		\forall x \in X
		,
		\\
		G(x) - G(\bar x) - \dual{x - \bar x}{\bar\xi}_{Y}
		&\ge
		\tilde c \min\set*{ \frac12 \norm{x - \bar x}_{X_2}^2, 1 }
		-
		\frac12 \sum_{j = 1}^n \dual{x - \bar x}{\hat \xi_j}_{Y}^2
		\quad
		\forall x \in X
		.
	\end{align*}
	Then,
	\begin{equation*}
		\sd^2 G_i (\bar x, \bar\xi; h)
		=
		\infty
		\qquad\forall h \in X_i \setminus X
	\end{equation*}
	for $i = 1, 2$
	and
	\begin{equation*}
		\sd^2 G_1(\bar x, \bar\xi; h)
		=
		\sd^2 G_2(\bar x, \bar\xi; h)
		\qquad\forall h \in X
		.
	\end{equation*}
	Moreover,
	$G_1$ is weak-$\star$ twice epidifferentiable at $\bar x$ for $\bar\xi$
	if and only if
	$G_2$ is weak-$\star$ twice epidifferentiable at $\bar x$ for $\bar\xi$.
\end{theorem}
\begin{proof}
	Let $h \in X_1$ be given such that
	$\sd^2 G_1 (\bar x, \bar\xi; h) < \infty$.
	By definition of the weak-$\star$ subderivative,
	for every $D > \sd^2 G_1 (\bar x, \bar\xi; h)$
	there exist sequences
	$\seq{t_k} \subset (0,\infty)$,
	$\seq{h_k} \subset X_1$
	such that $t_k \to 0$, $h_k \weaklystar h$ in $X_1$
	and
	\begin{equation*}
		\liminf_{k \to \infty}
		\frac{ G_1(\bar x + t_k h_k) - G_1(\bar x) - t_k \dual{h_k}{\bar\xi}_{Y} }{\frac12 t_k^2}
		\le
		D
		<
		\infty
		.
	\end{equation*}
	By passing to a subsequence (without relabeling),
	we can assume that this second-order difference quotient is actually convergent.
	Consequently, we can apply \cref{lem:limits_coincide}
	and this shows $h \in X$ and $h_k \weaklystar h$ in $X_2$.
	This shows the first claim for $i = 1$.
	Moreover, this also provides us with the inequality
	\begin{equation*}
		\sd^2 G_2(\bar x, \bar\xi; h)
		\le
		\liminf_{k \to \infty}
		\frac{ G_2(\bar x + t_k h_k) - G_2(\bar x) - t_k \dual{h_k}{\bar\xi}_{Y} }{\frac12 t_k^2}
		\le
		D.
	\end{equation*}
	Since $D > \sd^2 G_1(\bar x, \bar\xi; h)$ was arbitrary,
	this argument shows
	\begin{equation*}
		\sd^2 G_2(\bar x, \bar\xi; h)
		\le
		\sd^2 G_1(\bar x, \bar\xi; h)
	\end{equation*}
	for all $h \in X$ for which the right-hand side is finite
	and, consequently, for all $h \in X$.
	Symmetry settles the case $i = 2$ in the first assertion.
	Moreover, it provides us with the reverse inequality
	and this settles the second assertion.

	Now, let us assume that
	$G_1$ is weak-$\star$ twice epidifferentiable at $\bar x$ for $\bar\xi$.
	In order to show that $G_2$ satisfies the same property,
	let $h \in X_2$ be given.
	Clearly, it suffices to consider the case
	$\sd^2 G_2(\bar x, \bar\xi; h) < \infty$
	and, by the first part of the proof,
	this implies $h \in X \subset X_1$.
	For every sequence $\seq{t_k} \subset (0,\infty)$ with $t_k \to 0$,
	the definition of weak-$\star$ twice epidifferentiability
	provides us with a sequence $\seq{h_k} \subset X_1$
	such that
	\begin{equation*}
		\sd^2 G_2(\bar x, \bar\xi; h)
		=
		\sd^2 G_1(\bar x, \bar\xi; h)
		=
		\lim_{k \to \infty}
		\frac{ G_1(\bar x + t_k h_k) - G_1(\bar x) - t_k \dual{h_k}{\bar\xi}_{Y_1} }{\frac12 t_k^2}
		.
	\end{equation*}
	Due to \cref{lem:limits_coincide},
	we get $h_k \in X_2$ for $k$ large enough and $h_k \weaklystar h$ in $X_2$.
	W.l.o.g., we can assume $\seq{h_k} \subset X$
	(otherwise, we set $h_k = 0$ for all $k$ which are not large enough).
	Consequently, $\seq{h_k}$ also serves as a recovery sequence in the space $X_2$.
	This shows that
	$G_2$ is weak-$\star$ twice epidifferentiable at $\bar x$ for $\bar\xi$.
	By symmetry, this shows the final claim.
\end{proof}

This theorem provides a convenient way for the computation of
weak-$\star$ second subderivatives.
Indeed,
we can start with the computation of
the weak-$\star$ second subderivative in a nice space $X_1$.
Then, as long as the above strengthening of \eqref{eq:growth_of_G_without_F}
is satisfied in the space $X_1$ and in another space $X_2$,
\cref{thm:equivalent_subderivatives}
immediately gives us the expression for
the weak-$\star$ second subderivative in the space $X_2$.
This is precisely what has been done in \cite[Section~4.2]{Wachsmuth2026:1}.
Finally, we emphasize that the
additional conditions in \cref{thm:equivalent_subderivatives}
follow from convexity of $G$ and \eqref{eq:growth_of_G_without_F},
see \cref{lem:growth_of_G_Y}.
Moreover,
condition
\eqref{eq:growth_of_G_without_F} is needed anyway
for the application of the second-order theory
if $h \mapsto F''(\bar x) h^2$ is sequentially weak-$\star$ continuous,
see \cref{thm:SSC}.

\section*{Declaration on the use of AI tools}
All scientific content of this article was written by the author.
The AI model ``GLM 5.3 Flash''
was used for proofreading and for drafting this declaration.

\section*{Acknowledgement}
The author is grateful to Radu Ioan Boţ (Universität Wien)
for pointing out reference \cite{BartzBauschkeBorweinReichWang2007}.

The AI model ``GLM 5.3 Flash''
was accessed via
\href{https://chat-ai.academiccloud.de/}{CHAT AI}
provided by \href{https://gwdg.de}{GWDG}.

Funded by the Deutsche Forschungsgemeinschaft (DFG, German Research Foundation) via grant number 576024132.

\printbibliography

@ARTICLE{ChristofWachsmuth2017:1,
	title= {No-Gap Second-Order Conditions via a Directional Curvature Functional},
	author= {Christof, Constantin and Wachsmuth, Gerd},
	journal = {SIAM Journal on Optimization},
	volume = {28},
	number = {3},
	pages = {2097-2130},
	year = {2018},
	doi = {10.1137/17M1140418},
}

@ARTICLE{ChristofWachsmuth2017:3,
	author= {Christof, Constantin and Wachsmuth, Gerd},
	title= {Differential Sensitivity Analysis of Variational Inequalities with Locally {L}ipschitz Continuous Solution Operators},
	year= {2020},
	journal = {Applied Mathematics \& Optimization},
	doi = {10.1007/s00245-018-09553-y},
	volume = {81},
	number = {1},
	pages = {23--62},
}

@ARTICLE{BeerBorwein1990,
 title = {Mosco Convergence and Reflexivity},
 volume = {109},
 ISSN = {0002-9939},
 DOI = {10.2307/2048005},
 number = {2},
 journal = {Proceedings of the American Mathematical Society},
 publisher = {JSTOR},
 author = {Beer, Gerald and Borwein, Jonathan M.},
 year = {1990},
 pages={427}
}

@ARTICLE{Mosco1971,
 title = {On the continuity of the Young-Fenchel transform},
 volume = {35},
 ISSN = {0022-247X},
 DOI = {10.1016/0022-247x(71)90200-9},
 number = {3},
 journal = {Journal of Mathematical Analysis and Applications},
 publisher = {Elsevier BV},
 author = {Mosco, Umberto},
 year = {1971},
 pages={518–535}
}

@online{CorreaPerezArosSantander2025,
  author        = {Rafael Correa and Pedro Pérez-Aros and José Pablo Santander},
  title         = {Duality for the $\Gamma$-convergence of convex functions},
  year          = {2025},
  eprint        = {2509.21863},
  eprinttype    = {arXiv}
}

@article{WachsmuthWalter2024,
 title = {No-gap second-order conditions for minimization problems in spaces of measures},
 volume = {32},
 ISSN = {1262-3377},
 DOI = {10.1051/cocv/2026011},
 journal = {ESAIM: Control, Optimisation and Calculus of Variations},
 publisher = {EDP Sciences},
 author = {Wachsmuth, Gerd and Walter, Daniel},
 year = {2026},
 pages = {37}
}

@book{Attouch1984,
  title={Variational Convergence for Functions and Operators},
  author={Attouch, Hedy},
  isbn={9780273085836},
  series={Applicable Mathematics Series},
  year={1984},
  publisher={Pitman Advanced Pub. Program}
}

@article{WachsmuthWachsmuth2021,
  author        = {Daniel Wachsmuth and Gerd Wachsmuth},
  title         = {Second-order conditions for non-uniformly convex integrands: quadratic growth in $L^1$},
  year          = {2022},
  doi = {10.46298/jnsao-2022-8733},
  publisher = {Centre pour la Communication Scientifique Directe ({CCSD})},
  volume = {3},
  journal = {Journal of Nonsmooth Analysis and Optimization}
}

@ARTICLE{Harder2018,
	doi = {10.4171/zaa/1619},
	year = 2018,
	publisher = {European Mathematical Publishing House},
	volume = {37},
	number = {4},
	pages = {377--388},
	author = {Felix Harder},
	title = {Legendre Forms in Reflexive Banach Spaces},
	journal = {Zeitschrift für Analysis und ihre Anwendungen}
}

@PHDTHESIS{Harder2021,
  doi = {10.26127/BTUOPEN-5375},
  author = {Harder, Felix},
  title = {On bilevel optimization problems in infinite-dimensional spaces},
  school = {BTU  Cottbus-Senftenberg},
  year = {2021},
}

@ARTICLE{Rockafellar1970:2,
	author = {Rockafellar, Ralph Tyrrell},
	title = {{On the maximal monotonicity of subdifferential mappings}},
	journal = {{Pacific Journal of Mathematics}},
	issn = {0030-8730},
	volume = {33},
	pages = {209--216},
	year = {1970},
	publisher = {Mathematical Sciences Publishers (MSP), Berkeley, CA; Pacific Journal of Mathematics c/o University of California, Berkeley, CA},
	doi = {10.2140/pjm.1970.33.209}
}

@ARTICLE{Rockafellar1966,
 title = {Characterization of the subdifferentials of convex functions},
 volume = {17},
 ISSN = {0030-8730},
 DOI = {10.2140/pjm.1966.17.497},
 number = {3},
 journal = {Pacific Journal of Mathematics},
 publisher = {Mathematical Sciences Publishers},
 author = {Rockafellar, Ralph Tyrrell},
 year = {1966},
 pages={497–510}
}

@ARTICLE{BartzBauschkeBorweinReichWang2007,
 title = {Fitzpatrick functions, cyclic monotonicity and Rockafellar’s antiderivative},
 volume = {66},
 ISSN = {0362-546X},
 DOI = {10.1016/j.na.2006.01.013},
 number = {5},
 journal = {Nonlinear Analysis: Theory, Methods \& Applications},
 publisher = {Elsevier BV},
 author = {Bartz, Sedi and Bauschke, Heinz H. and Borwein, Jonathan M. and Reich, Simeon and Wang, Xianfu},
 year = {2007},
 pages={1198–1223}
}

@ARTICLE{Do1992,
	author = {Chi Ngoc Do},
	TITLE = {Generalized second-order derivatives of convex functions in reflexive {B}anach spaces},
	JOURNAL = {Transactions of the American Mathematical Society},
	VOLUME = {334},
	YEAR = {1992},
	NUMBER = {1},
	PAGES = {281--301},
	DOI = {10.2307/2153983}
}

@BOOK{HuPapageorgiou1997,
  title     = {Handbook of multivalued analysis},
  author    = {Hu, Shouchuan and Papageorgiou, Nikolas S},
  publisher = {Springer},
  series    = {Mathematics and Its Applications},
  year      =  1997,
  address   = {Dordrecht, Netherlands},
}

@ARTICLE{BrondstedRockafellar1965,
 title = {On the subdifferentiability of convex functions},
 volume = {16},
 ISSN = {1088-6826},
 DOI = {10.1090/s0002-9939-1965-0178103-8},
 number = {4},
 journal = {Proceedings of the American Mathematical Society},
 publisher = {American Mathematical Society (AMS)},
 author = {Brøndsted, Arne and Rockafellar, Ralph Tyrrell},
 year = {1965},
 pages = {605--611}
}

@ARTICLE{IvanovZlateva2008,
 title = {A new proof of the integrability of the subdifferential of a convex function on a Banach space},
 volume = {136},
 ISSN = {0002-9939},
 DOI = {10.1090/s0002-9939-08-09178-8},
 number = {5},
 journal = {Proceedings of the American Mathematical Society},
 publisher = {American Mathematical Society (AMS)},
 author = {Ivanov, Milen and Zlateva, Nadia},
 year = {2008},
 pages={1787–1793}
}

@ARTICLE{BorchardWachsmuth2023,
 title = {Second-order conditions for spatio-tem\-po\-ral\-ly sparse optimal control via second subderivatives},
 volume = {5},
 DOI = {10.46298/jnsao-2024-12604},
 journal = {Journal of Nonsmooth Analysis and Optimization},
 publisher = {Centre pour la Communication Scientifique Directe (CCSD)},
 author = {Borchard, Nicolas and Wachsmuth, Gerd},
 year = {2024}
}

@online{BorchardMeyerWachsmuth2026:1,
  author        = {Nicolas Borchard and Christian Meyer and Gerd Wachsmuth},
  title         = {No-gap second-order conditions for optimization problems involving transport distances},
  year          = {2026},
  eprint        = {2607.28264},
  eprinttype    = {arXiv}
}

@online{WachsmuthWalter2025:1,
  author        = {Gerd Wachsmuth and Daniel Walter},
  title         = {Proximal gradient methods in Banach spaces},
  year          = {2025},
  eprint        = {2509.24685},
  eprinttype    = {arXiv}
}

@online{Wachsmuth2026:1,
  author        = {Gerd Wachsmuth},
  title         = {Second-order conditions for bang-bang control of elliptic equations in arbitrary dimensions},
  year          = {2026},
  eprint        = {2602.14632},
  eprinttype    = {arXiv}
}

@ARTICLE{Rockafellar1985,
 title = {Maximal monotone relations and the second derivatives of nonsmooth functions},
 volume = {2},
 ISSN = {1873-1430},
 DOI = {10.1016/s0294-1449(16)30401-2},
 number = {3},
 journal = {Annales de l’Institut Henri Poincaré C, Analyse non linéaire},
 publisher = {European Mathematical Society - EMS - Publishing House GmbH},
 author = {Rockafellar, Ralph Tyrrell},
 year = {1985},
 pages={167–184}
}

@ARTICLE{Noll1995,
 title = {Directional differentiability of the metric projection in Hilbert space},
 volume = {170},
 ISSN = {0030-8730},
 DOI = {10.2140/pjm.1995.170.567},
 number = {2},
 journal = {Pacific Journal of Mathematics},
 publisher = {Mathematical Sciences Publishers},
 author = {Noll, Dominikus},
 year = {1995},
 pages={567–592}
}

@ARTICLE{Levy1999,
 title = {Sensitivity of Solutions to Variational Inequalities on Banach Spaces},
 volume = {38},
 ISSN = {1095-7138},
 DOI = {10.1137/s036301299833985x},
 number = {1},
 journal = {SIAM Journal on Control and Optimization},
 publisher = {Society for Industrial & Applied Mathematics (SIAM)},
 author = {Levy, Adam B.},
 year = {1999},
 pages={50–60}
}

@ARTICLE{AttouchBeer1993,
 title = {On the convergence of subdifferentials of convex functions},
 volume = {60},
 ISSN = {1420-8938},
 DOI = {10.1007/bf01207197},
 number = {4},
 journal = {Archiv der Mathematik},
 publisher = {Springer Science and Business Media LLC},
 author = {Attouch, Hedy and Beer, Gerald},
 year = {1993},
 pages={389–400}
}

@ARTICLE{CombariThibault1998,
 title = {On the graph convergence of subdifferentials of convex functions},
 volume = {126},
 ISSN = {0002-9939},
 DOI = {10.1090/s0002-9939-98-04724-8},
 number = {8},
 journal = {Proceedings of the American Mathematical Society},
 publisher = {American Mathematical Society (AMS)},
 author = {Combari, C. and Thibault, L.},
 year = {1998},
 pages={2231–2240}
}

@BOOK{BonnansShapiro2000,
	title                    = {Perturbation Analysis of Optimization Problems},
	author                   = {J. Fr{\'{e}}d{\'{e}}ric Bonnans and Alexander Shapiro},
	publisher                = {Springer},
	year                     = {2000},
	doi = {10.1007/978-1-4612-1394-9},
	address                  = {Berlin}
}

\end{document}